\documentclass{llncs}

\RequirePackage{amsmath}
\usepackage{amssymb}
\usepackage{amsmath}
\usepackage{mathabx}
\usepackage{graphicx}
\usepackage{color}
\usepackage{algorithm,algpseudocode}
\usepackage{todonotes}

\begin{document}

\title{Guaranteed  optimal reachability control of reaction-diffusion equations 
using  one-sided Lipschitz constants and model reduction}
\author{A. Le Co\"ent \inst{1}
\and
L. Fribourg \inst{2}
}

\institute{Department of Computer Science, Aalborg University \\
Selma Largerl\o fs Vej 300, 9220 Aalborg, Denmark \\
\email{adrien.le-coent@ens-cachan.fr} \and
LSV, ENS Paris-Saclay, CNRS, Universit\'e Paris Saclay \\
91 Avenue du Pr\'esident Wilson, 94235 Cachan Cedex, France \\
\email{fribourg@lsv.fr}}
  

\graphicspath{{./figures/}}

\titlerunning{Distributed control synthesis using Euler's method}

\authorrunning{A. Le Co\"ent, 
  L. Fribourg} 

\maketitle



\begin{abstract}
We show that, for any spatially discretized system  of 
reaction-diffusion, the 
approximate solution given by the explicit 
Euler time-discretization scheme  converges to the exact 
time-continuous solution, provided that diffusion coefficient
be sufficiently large. By ``sufficiently large'', we mean that the 
diffusion coefficient value makes the 
{\em one-sided Lipschitz constant} of
the reaction-diffusion system negative. We apply this 
result to solve a finite horizon control problem for a 1D
reaction-diffusion example. We also explain how to perform
{\em model reduction} in order to improve the efficiency of the method.
\end{abstract}


\section{Introduction}
\subsection{Guaranteed reachability analysis}

Given a system of Ordinary Differential equations (ODEs) of dimension $n$
satisfying standard conditions of existence and uniqueness of the solution, the area of 
{\em Numerical Analysis} makes use of numerical tools in order to
compute the approximate value of the solution, starting at an initial point of $\mathbb{R}^n$,
with high accuracy:
1st order methods (explicit/implicit Euler method, trapezoid rule), higher-order Runge-Kutta methods, etc.
In contrast, the area of {\em Guaranteed (or Symbolic) Analysis} is devoted to the construction of an overapproximation of the set of solutions that start, not at a single point of $\mathbb{R}^n$, but from a dense compact set of initial 
points. Guaranteed analysis, in its modern form, has been initiated in the 60's by R.E. Moore and his creation of {\em Interval Arithmetic} \cite{Moore66}: the set of solutions
(or trajectories) are overapproximated by a sequence of 
``rectangular sets'', i.e., cross-product of intervals of $\mathbb{R}$. A set of arithmetic and differential
calculus has been created for manipulating such  sets. 
An overapproximation of the set of trajectories is computed
using a Taylor development up to 
some order and an overestimation 
of the ``Lagrange remainder''. The method has been 
considerably refined  in the 90's \cite{BerzH98,BerzM98,Lohner87,Nedialkov99,NedialkovKS04}.
These recent techniques make use of different convex data structures
such as parallelepipeds \cite{Lohner87}
or zonotopes \cite{girard2005reachability,Kuhn:1998:RCO:287056.287083}
instead of rectangular sets in order to enclose the flow of ODEs.
%

Such methods are typically applied to the formal {\em proof} of correctness of
ODE integration, and more generally, to guarantee that the solutions
of the ODEs satisfy some desired properties. 
%
Guaranteed reachability analysis generally treats {\em linear} systems.
Extensions to nonlinear systems have been proposed, e.g., in
\cite{AlthoffSB08}, using
local linearizations 
(see also \cite{MitchellBT01,MitchellT03}).

\subsection{Guaranteed optimal control}

In presence of inputs, we can use guaranteed analysis to describe
a law that allows the
system to satisfy a desired property. 
This corresponds to the topic of
{\em guaranteed (or correct-by-design) control synthesis}.
Several works have reecently applied guaranteed analysis to 
{\em optimal} control synthesis.
Thus, in \cite{SchurmannA17,SchurmannKA18}, the authors
focus on a (finite
time-horizon) optimal control procedure with a formal
guarantee of safety constraint satisfaction, using zonotopes
as
state set representations.
%
In \cite{Estrela17}, the authors focus on (periodically) sampled
systems, and perform 
reachability analysis 
using convex
polytopes as state set representations. 
%
In \cite{kapela2009lohner,maidens2014reachability,fan2018simulation,ReissigR19,RunggerR17}, 
the authors
construct 
an over-approximation of the set
of trajectories 
using a growth bound
(bounding the distance of neighboring trajectories)
exploiting the notion of {\em one-sided Lipschitz constant}
(also called ``logarithmic norm'' or ``matrix norm'').
The notion of ``one-sided Lipschitz (OSL) constant'' has been introduced
independently by Dahlquist  \cite{dahlquist1958stability} and Lozinskii \cite{lozinskii1958error} in order to derive error bounds in initial value problems (see survey in \cite{soderlind2006logarithmic}).
We used ourselves OSL constants 
in the context of symbolic optimal control in \cite{LeCoentF19}. 
The main difference with  previous work
 \cite{kapela2009lohner,maidens2014reachability,fan2018simulation,ReissigR19,RunggerR17}
is that our method makes use of  explicit Euler's algorithm 
for ODE integration (cf. \cite{AdrienRP17,SNR17}) instead of
sophisticated algorithms  such as 
Lohner's algorithm \cite{kapela2009lohner}
or interval Taylor series methods \cite{NedialkovKS04}. 
This leads us to a simple
implementation of just
a few hundred lines of Octave
(see \cite{oslator}).

As explained in \cite{Saluzzi18}, 
using the Dynamic Programming (DP) \cite{Bellman57}
one can approximate the ``value'' of the solution of
Hamilton-Jacobi-Bellman (HJB) equations. 
In \cite{Falcone1999,Saluzzi18}, 
the authors thus show how to use
finite difference schemes,
Euler time integration and DP
for solving finite horizon control problems.
Furthermore, they give {\em a priori} errors estimates which
are first-order in the size $\Delta t$ of the time discretization step; 
however, the error involves a constant $C(T)$ 
which depends {\em exponentially}
on the length $T$ of the finite horizon\footnote{$C(T)=O(e^{L_fT})$ where
$L_f$ is the Lipschitz constant associated with vector field $f$.}.
We solve here finite horizon control problems along the same lines
(using finite difference, explicit Euler and DP) but, 
under the hypothesis of OSL {\em negativity} (see section \ref{ss:rde}), 
we obtain an
error upper bound that is {\em linear} in $T$
(see Section~\ref{ss:error}, Theorem \ref{th:2}).


\subsection{Reaction-diffusion equations}\label{ss:rde}

It is natural to adapt 
the optimal control methods of ODEs 
to the control of 
Partial Differential Equations (PDEs).
This can be done by transforming
the PDE into (a vast system of) ODEs, using space discretization
techniques such as finite difference or finite element methods.
%
%
In the present work, we focus on
a particular class of 
{\em non-linear} PDEs called ``reaction-diffusion'' equations.
Reaction-diffusion equations cover a variety of particular cases with important
applications in mathematical physics, and in biological models such as
the Schl\"ogl model or the FitzHugh-Nagumo system \cite{CasasRT18}.
The problem of optimal control
of reaction-diffusion equations has been recently
the topic of many works of (classical) numerical analysis: see, e.g.,
\cite{Barthel2010,CourtKP18,Finotti12,GriesseV05,MouraF11,MouraF13}.
%
%

The notion OSL constant can be naturally extended to 
PDEs and reaction-diffusion equations in particular, as shown in 
\cite{arcak2011certifying,aminzare2013logarithmic,aminzare2014guaranteeing,aminzare2016some}.
In these works, the authors focus on the case where the OSL constant associated with the 
reaction-diffusion equation is {\em negative}. In this case, the system has 
a {\em contractivity}  (or ``incremental stability'') property which expresses the fact that all solutions converge exponentially to each other (see \cite{sontag2010contractive}).

In this work, we also study reaction-diffusion equations with negative OSL constants, but the equations are equipped with {\em control inputs}, and  
the problem of controlling these inputs in an optimal way is here considered.

\subsection{Model reduction}

In order to reduce the large dimension of ODE systems originating
from the PDE space discretization, Model Order Reduction (MOR) techniques 
are often used in conjunction with the analysis of ODE systems.
The idea is to first infer the 
optimal control at a {\em reduced level}, then apply it at the original level.
In the field of guaranteed analysis, the MOR technique
of ``balanced truncation'' was used to
treat {\em linear} systems
(e.g., \cite{Althoff17,HanK04,HanK06,Syncop15}).
In~\cite{kalisek14}, a MOR technique based on spectral element method
was coupled to an HJB approach
for application to advection-reaction-diffusion systems
(cf. \cite{KaliseK18} for application to
semilinear parabolic PDEs).
The MOR technique of ``Proper 
Orthogonal Decomposition (POD)''  was
coupled to an HJB approach
in \cite{AllaFV17,AllaS19,KunischVX04}.
%
%
%
Here, we couple our  HJB-based method to
a simple {\em ad hoc} reduction  method (see Section~\ref{ss:MOR}).\\

The plan of the paper is as follows:
We explain how to convert the reaction-diffusion equation
into a system of ODEs
by domain discretization in
Section~\ref{ss:space}, and how to approximate the solution of the latter system
using the explicit Euler scheme of time integration in
Section \ref{ss:Euler}.
Our procedure for solving finite horizon
control problems is explained
in Section \ref{ss:approx}.
In Section
\ref{ss:error}, we give an upper bound to the error between the approximate
value thus computed and the exact optimal value.
In Section \ref{ss:MOR}, we explain how to perform MOR in order to treat systems of larger dimension.
We conclude in Section \ref{sec:final}.

\section{Optimal Reachability Control of Reaction-Diffusion Equations}
Let us consider the special 
class of PDEs called ``reaction-diffusion'' equations. 
For the sake of notation simplicity, we focus on 1D reaction-diffusion 
equations with Dirichlet boundary conditions
(the domain $\Omega$ is of the form $[0,L]\subset \mathbb{R}$), but the method applies to 
2D or 3D reaction-diffusion equations with other boundary conditions.
A 1D {\em reaction-diffusion} system with Dirichlet boundary conditions
is of the form:
\begin{align*}
& \frac{\partial {\bf y}(t,x)}{\partial t} =\sigma\frac{\partial^2 {\bf y}(t,x)}{\partial x^2} + f({\bf y}(t,x)), \ \ \ t\in [0,T],\ x\in \Omega\equiv [0,L].
\\
& {\bf y}(t,0) = u_0(t), \quad {\bf y}(t,L) = u_L(t), \ \ \ t\in [0,T],
\\
& {\bf y}(0,x) = {\bf y}_0(x),\ \ \ x\in \Omega\equiv [0,L].
\end{align*}
Here, ${\bf y}={\bf y}(t,x)$ is an $\mathbb{R}$-valued unknown function,
$\Omega$ is a bounded domain in $\mathbb{R}$ with boundary $\partial\Omega :=\{0,L\}$,
and $f$ is a function from $[0,T]\times\Omega$ to $[0,1]$.
Also ${\bf y}_0(x)$ is a given function called
``initial condition'', and $\sigma$ a positive constant,
called ``diffusion constant''.

The boundary control $u(\cdot) := (u_0(\cdot), u_L(\cdot))$ 
that we consider here,
is a {\em piecewise constant} (or ``staircase'')
function from $[0,T]$ to a {\em finite} set $U\subset [0,1]\times[0,1]$.The 
control $u(t)$ changes
its value {\em periodically} at $t=\tau,2\tau,\dots$.
We assume that $T=k\tau$ for some positive integer $k$.
The constant $\tau$ is called the ``switching (or sampling) period''.

Given an initial condition ${\bf y}_0(\cdot)$ such that ${\bf y}_0(x)\in [0,1]$ for all 
$x\in [0,L]$,
we assume that, for any boundary control $u(\cdot)$, 
the solution ${\bf y}(\cdot,\cdot)$  of the system exists, is unique,
and ${\bf y}(t,x)\in[0,1]$ for all $(t,x)\in [0,T]\times [0,L]$.

\subsection{Domain discretization} \label{ss:space}
A well-known approach in numerical analysis of PDEs 
(see, e.g., \cite{Koto08})
is to discretize in space
by finite difference or finite element methods in order to transform 
the PDE into a system of ODEs. 

Let $M$ be a positive integer, $h=L/(M+1)$, and let $\Omega_h$ be a uniform
grid with nodes $x_j=jh$, $j=1,\dots,M$. By replacing the 2nd order spatial
derivative with the second order centered difference, we obtain 
a space-discrete  approximation:
$$\frac{dy}{dt}=\sigma{\cal L}_h y+\sigma\varphi_h(t,u)+f(t,y),$$

with $y(t)=[y^1(t),\dots,y^{M}(t)]^T$, $y^j(t)\approx {\bf y}(t,x_j)$, and
\[
{\cal L}_h=\frac{1}{h^2}
  \begin{bmatrix}
    -2 & 1 & 0 & \cdots & 0 \\
    1 & -2 & 1 &\cdots & 0\\
    0 & 1 & -2 & \cdots & 0\\
    \ & \ & \cdots &\ &\ \\
    0 & 0 & \cdots & 1 & -2
  \end{bmatrix}
\]

$$\varphi_h(t,u)=\frac{1}{h^2}[u_0(t),0,\dots,0,u_L(t)]^\top.$$

The point $y(t)$, often abbreviated as $y$, is thus an element of $S=[0,1]^{M}$.
\subsection{Explicit Euler time integration}\label{ss:Euler}
Let us abbreviate the equation 
$$\frac{dy}{dt}=\sigma{\cal L}_h y+\sigma\varphi_h(t,u)+f(t,y)$$
by:
$$\frac{dy}{dt}=f_u(t,y).$$
We denote by $Y_{t,y_0}^u$, the solution 
$y$ 
of the system at time~$t\in[0,\tau)$ 
controlled by mode $u\in U$, for initial condition~$y_0$.
Given a sequence of modes (or ``pattern'') $\pi := u_k\cdots u_1\in U^k$, we denote by
$Y_{t,y_0}^{\pi}$
the solution of the system for  mode $u_k$ on $t\in [0,\tau)$
with initial condition~$y_0$,
extended continuously with the solution of the system for mode $u_{k-1}$ on $t\in[\tau,2\tau)$, and so on iteratively until mode $u_1$ on $t\in[(k-1)\tau,k\tau]$.

Let us now approximate the solution of the system by performing time integration with the
{\em explicit Euler} scheme. This yields:
$$y_{n+1}=y_n+\tau f_u(t_n,y_n),$$
%
Here $y_n$ is an approximate value of $y(t_n)$.
%
Given  a starting point $z\in {\cal X}$ and a mode $u\in U$, 
we denote 
by  $\tilde{Y}_{t,z}^u$ the Euler-based image of~$z$ at time $t$ via $u$
for $t\in[0,\tau)$.
We have:
$\tilde{Y}_{t,z}^u := z+ t\ f_u(z).$
We denote similarly
by $\tilde{Y}^\pi_{t,z}$ the Euler-based image of~$z$ via 
pattern $\pi\in U^k$ at time $t\in [0,k\tau]$.


\subsection{Finite horizon control problems}\label{ss:approx}


Let us now explain the principle of the method of optimal control
of ODEs used in~\cite{LeCoentF19}, in the present context.
We consider the {\em cost function}: $J_{k}:[0,1]^M\times U^k\rightarrow \mathbb{R}_{\geq 0}$ 
defined by:
$$J_{k}(y,\pi)=\|Y_{k\tau,y}^{\pi}-y_f\|,$$
where 
$\|\cdot\|$ denotes the Euclidean norm in $\mathbb{R}^M$,
and
$y_f\in[0,1]^M$ is a given ``target'' 
state.

We consider the {\em value function} ${\bf v}_k:[0,1]^M\rightarrow \mathbb{R}_{\geq 0}$
defined by:
$${\bf v}_k(y) := \min_{\pi\in U^k}\{J_{k}(y,\pi)\}\equiv
\min_{\pi\in U^k}\{\|Y_{k\tau,y}^{\pi}-y_f\|\}.$$ 

Given $k\in\mathbb{N}$ and $\tau\in\mathbb{R}_{>0}$, we consider the following {\em finite time horizon optimal control problem}: 
 Find for each $y\in[0,1]^M$
\begin{itemize}
\item the {\em value} 
${\bf v}_k(y)$, i.e.
%
$$\min_{\pi\in U^k}\{\|Y_{k\tau,y}^{\pi}-y_f\|\},$$

\item and an {\em optimal pattern}:
$$\pi_k(y) := arg\min_{\pi\in U^k}\{\|Y_{k\tau,y}^{\pi}-y_f\|\}.$$
\end{itemize}


In order to solve such optimal control problems, 
a classical ``direct'' method consists in
{\em spatially discretizing} the state space $S=[0,1]^M$
(i.e., the space of values of $y$).
We consider here a uniform partition of~$S$ into a finite number $N$
of cells of equal size:
in our case , this means that interval $[0,1]$ is divided into
$K$ subintervals of equal size, and $N=K^M$. A cell 
thus corresponds  to a $M$-tuple of subintervals. The center of a cell
coresponds to the $M$-tuple of the subinterval midpoints.
The associated grid ${\cal X}$ is the set of centers
of the cells of~$S$. The center~$ z\in{\cal X}$ of a cell $C$ is considered as the $\varepsilon$-{\em representative} of 
all the points of $C$. We suppose that
the  cell size is such that $\|y - z\|\leq\varepsilon$,
for all $y\in C$ (i.e. $K\geq \sqrt{M}/2\varepsilon$).
%
In this context, the direct method proceeds as follows
(cf. \cite{LeCoentF19}): we consider the points of ${\cal X}$
as the vertices of a finite oriented graph;
there  is a connection from $ z\in {\cal X}$ 
to $ z'\in {\cal X}$ if $ z'$ is the $\varepsilon$-representative
of the Euler-based image $(z +\tau f_u( z))$ of~$z$, for some $u\in U$.
We then compute using dynamic programming
the ``path of length $k$ with minimal cost'' starting at $ z$: 
such a path is a sequence of $k+1$ 
connected points $ z\  z_k\  z_{k-1}\ \cdots\  z_1$ of
${\cal X}$ which minimizes the distance $\| z_1-y_f\|$.
This procedure allows us to compute
a pattern $\pi^\varepsilon_k(z)$ of length
$k$, which approximates the optimal pattern $\pi_k(y)$.

\begin{definition}
The function $next^{u}: {\cal X}\rightarrow {\cal X}$ 
is defined by:
\begin{itemize}
\item 
    $next^{u}( z)= z'$, where
    $ z'$ 
    is the $\varepsilon$-representative of $\tilde{Y}_{\tau, z}^{u}$.
\end{itemize}
\end{definition}
\begin{definition}
For all point $ x\in {\cal X}$, the {\em spatially discrete
value function} ${\bf v}^{\varepsilon}_k:{\cal X}\rightarrow \mathbb{R}_{\geq 0}$ 
is defined by:
\begin{itemize}
\item for $k=0$, ${\bf v}_k^{\varepsilon}( z)=\| z-y_f\|$,
\item for $k\geq 1$,
%
    ${\bf v}_{k}^{\varepsilon}( z)=\min_{u\in U}\{{\bf v}_{k-1}^{\varepsilon}(next^u( z))\}$.
\end{itemize}
\end{definition}
%
\begin{definition}
The {\em approximate  optimal pattern of length $k$}
associated
to $z\in{\cal X}$,
denoted by $\pi_k^{\varepsilon}( z)\in U^k$,  is 
defined by:
\begin{itemize}
\item if $k=0$, $\pi_k^{\varepsilon}( z)=\mbox{nil}$,
\item if $k\geq 1$, $\pi_k^{\varepsilon}( z) = {\bf u}_k( z) \cdot \pi'$
where
$${\bf u}_{k}( z)=arg \min_{u\in U}\{{\bf v}_{k-1}^{\varepsilon}(next^u( z))\}$$
and $\pi' =\pi_{k-1}^{\varepsilon}( z')$
\ \ with \ $ z'=next^{{\bf u}_k( z)}( z)$.
\end{itemize}
\end{definition}
%
It is easy to
construct
a procedure $PROC_k^{\varepsilon}$ which takes a point $ z\in {\cal X}$ as input, and returns 
an approximate optimal pattern~$\pi_k^{\varepsilon}\in U^k$.
%

\begin{remark}
The complexity of  $PROC_k^\varepsilon$ is  
$O(m\times k\times N)$
where $m$ is the number of modes ($|U|=m$), 
$k$ the time-horizon length ($T=k\tau$)
and $N$ the number of cells of ${\cal X}$
($N = K^M$ with $K=\sqrt{M}/2\varepsilon$).
\end{remark}

%
%

\subsection{Error upper bound}\label{ss:error}
Given a point $y\in S$ of $\varepsilon$-representative $z\in {\cal X}$,
and a pattern $\pi^\varepsilon_k$ returned by~$PROC_k^\varepsilon(z)$, we are now going to show that the distance
$\|\tilde{Y}_{k\tau,z}^{\pi^\varepsilon_k}-y_f\|$ converges to~${\bf v}_k(y)$
as $\varepsilon\rightarrow 0$.
We first consider the ODE:
$\frac{dy}{dt}=f_u(y)$, and give an upper bound to the error between
the exact solution of the ODE and
its Euler approximation (see \cite{SNR17}).
\begin{definition}\label{def:delta}
Let $\mu$ be a given positive constant. Let us define, for all 
$u\in U$ and $t\in [0,\tau]$,
$\delta^u_{t,\mu}$ as follows:
$$\mbox{if } \lambda_u <0:\ \ 
\delta^u_{t,\mu}=\left(\mu^2 e^{\lambda_u t}+
 \frac{C_u^2}{\lambda_u^2}\left(t^2+\frac{2 t}{\lambda_u}+\frac{2}{\lambda_u^2}\left(1- e^{\lambda_u t} \right)\right)\right)^{\frac{1}{2}}$$
%
$$\mbox{if } \lambda_u = 0:\ \ 
\delta^u_{t,\mu}= \left( \mu^2 e^{t} + C_u^2 (- t^2 - 2t + 2 (e^t - 1)) \right)^{\frac{1}{2}}$$
%
$$\mbox{if } \lambda_u > 0:\ \ 
\delta^u_{t,\mu}=\left(\mu^2 e^{3\lambda_u t}+ 
\frac{C_u^2}{3\lambda_u^2}\left(-t^2-\frac{2t}{3\lambda_u}+\frac{2}{9\lambda_u^2}
\left(e^{3\lambda_u t}-1\right)\right)\right)^{\frac{1}{2}}$$
%
where $C_u$ and $\lambda_u$ are real constants specific to function $f_u$,
defined as follows:
$$C_u=\sup_{y\in S} L_u\|f_u(y)\|,$$
where $L_u$ denotes the Lipschitz constant for $f_u$, and
$\lambda_u$ is the OSL constant associated to $f_u$, i.e., the 
minimal constant such that, for all $y_1,y_2\in S$:
$$\langle f_u(y_1)-f_u(y_2), y_1-y_2\rangle \leq \lambda_u\|y_1-y_2\|^2,$$
where $\langle\cdot,\cdot\rangle$ denotes the scalar product of two vectors
of $S$.
\end{definition}

\begin{proposition}\label{prop:basic}\cite{SNR17}
Consider the solution $Y_{t,y_0}^u$ of $\frac{dy}{dt}=f_u(y)$ with
initial condition~$y_0$ of $\varepsilon$-representative $z_0$
(hence such that $\|y_0-z_0\|\leq\varepsilon$), 
and the approximate
solution $\tilde{Y}_{t,z_0}^u$ given by the explicit Euler scheme.
For all $t\in[0,\tau]$, we have:
$$\|Y_{t,y_0}^u-\tilde{Y}_{t,z_0}^u\|\leq \delta^u_{t,\varepsilon}.$$
\end{proposition}
\begin{proposition}
Consider the system $\frac{dy}{dt}=f_u(y)$
with $f_u(y) := \sigma{\cal L}_hy+\sigma\varphi_h(t,u)+f(y)$.
For a diffusion coefficient $\sigma>0$ 
sufficiently large, the OSL constant~$\lambda_u$ associated
to $f_u$ is such that: $\lambda_u <0$.
\end{proposition}
\begin{proof}
Consider the ODE:
$\frac{dy}{dt}=f_u(y)=\sigma{\cal L}_hy+\sigma\varphi_h(t,u)+f(y)$.
For all $y_1,y_2\in S$, we have: $\langle f(y_2)-f(y_1), y_2-y_1\rangle\leq \lambda_f \|y_2-y_1\|^2$,
where $\lambda_f$ is the OSL constant of $f$.
%
Hence:
\begin{align*}\langle f_u(y_2)-f_u(y_1),y_2-y_1\rangle&=\langle \sigma{\cal L}_h(y_2-y_1) + f(y_2)-f(y_1),y_2-y_1\rangle \\ &\leq
(y_2-y_1)^\top(\sigma{\cal L}_h+\lambda_f)(y_2-y_1).\end{align*}
Since 
$y^\top{\cal L}_hy<0$ for all $y\in S$
(negativity of the quadratic form associated to~${\cal L}_h$), 
we have:
$$\lambda_u\|y_1-y_2\|^2 
\leq (y_2-y_1)^T(\sigma{\cal L}_h+\lambda_f)(y_2-y_1) <0,$$
for $\sigma>0$ sufficiently large. Hence $\lambda_u<0$.
\hspace*{\fill} $\Box$
\end{proof}

\begin{lemma}\label{lemma:1}
Consider the system $\frac{dy}{dt}=f_u(y)$
where the OSL constant $\lambda_{u}$ associated to $f_u$ is negative,
and initial error $e_0:=\|y_0-z_0\|>0$. 
Let
$G_u:=\frac{\sqrt{3}e_0|\lambda_u|}{C_u}$.
Consider the (smallest) positive root
$$\alpha_u := 1+|\lambda_{u}|G_{u}/4-\sqrt{1+(\lambda_{u}G_{u}/4)^2}$$
of equation:
$-\frac{1}{2}|\lambda_{u}|  G_{u}
+(2+\frac{1}{2}|\lambda_{u}|G_{u})\alpha-\alpha^2 = 0.$

Suppose:
$\frac{|\lambda_u|G_u}{4}<1.$
Then we have $0<\alpha_u< 1$, and, for all $t\in[0,\tau]$
with $\tau\leq G_u(1-\alpha_u)$:
$$\delta_{e_0}^u(t) 
\leq e_0.$$ 
\end{lemma}
\begin{proof}
See Appendix 1.
\end{proof}
\begin{remark}
In practical case studies $|\lambda_u|$ is often small, and the term  $(\lambda_{u}G_{u}/4)^2$ can be neglected, leading to $\alpha_u\approx |\lambda_u|G_u/4$
and $G_u(1-\alpha_u)\approx G_u(1-\frac{|\lambda_u|G_u}{4})\approx G_u$.
\end{remark}
\begin{remark}
It follows that, for $\tau\leq G_u(1-\alpha_u)$,
the Euler explicit scheme is {\em stable}, in the sense that initial errors are damped out. 
\end{remark}
\begin{remark}\label{rk:subsampling}
If $\tau > G_u(1-\alpha_u)$, we can make use of {\em subsampling}, i.e.,
decompose~$\tau$ into a sequence of elementary time steps $\Delta t$
with $\Delta t \leq G_u(1-\alpha_u)$ in order to be still able to apply
Lemma~\ref{lemma:1} (see Example~\ref{ex:1}).
Let us point out that Lemma \ref{lemma:1} (and the use of subsampling) allows to ensure set-based reachability with the use of procedure $PROC_k^\varepsilon$. Indeed, in this setting, the explicit Euler scheme leads to decreasing errors, and thus, point based computations performed with the center of a cell can be applied to the entire cell. 
\end{remark}

We suppose henceforth that the system $\frac{dy}{dt}=f_u(y)$
satisfies:
$$(H):\ \ \ \lambda_u<0,\ \frac{|\lambda_u|G_u}{4}<1\  \mbox{ and }\  \tau \leq G_u(1-\alpha_u),\ \mbox{ for all } u\in U.$$
%
From Proposition~\ref{prop:basic} and Lemma~\ref{lemma:1}, it easily follows:
\begin{theorem}\label{th:1}
Consider a system $\frac{dy}{dt}=f_u(y)$ satisfying $(H)$,
and a point $y\in S$ of $\varepsilon$-representative
$z\in {\cal X}$.
We have:
$$\|Y_{t,y}^\pi-\tilde{Y}_{t,z}^\pi\|\leq \varepsilon,\ \ \ 
\mbox{ for all }\ \pi\in U^k \mbox{ and } t\in[0,k\tau].$$
\end{theorem}

\begin{proposition}\label{prop:inf}
Let $z\in{\cal X}$ and
$\pi_k^\varepsilon$ be the pattern of $U^k$ returned by $PROC_k^\varepsilon(z)$.
For all $\pi\in U^k$, we have:
$$\|\tilde{Y}_{k\tau,z}^{\pi_k^\varepsilon}-y_f\|\leq 
\|\tilde{Y}_{k\tau,z}^{\pi}-y_f\| +  2k\varepsilon.$$
\end{proposition}
\begin{proof}
W.l.o.g., let us suppose that $y_f$ is the origin $O$.
Let us prove by induction on $k$: 
$$\|\tilde{Y}_{k\tau,z}^{\pi_k^\varepsilon}\|\leq 
\|\tilde{Y}_{k\tau,z}^{\pi}\| +2k\varepsilon.$$
Let $\pi_k^\varepsilon := u_k\cdots u_1$. 
The base case $k=1$ is easy. For $k\geq 2$,
we have:

\ \ $\|\tilde{Y}_{k\tau,z}^{\pi_k^\varepsilon}\|=\|\tilde{Y}_{(k-1)\tau,z_k}^{u_{k-1}\cdots u_1}\|$ with $z_k=\tilde{Y}_{\tau,z}^{u_k}$
with $u_k=argmin_{u\in U}\{{\bf v}_{k-1}^\varepsilon(next^u(z))\}$

$\leq \|\tilde{Y}_{(k-1)\tau,next^{u_k}(z_k)}^{u_{k-1}\cdots u_1}\|+\varepsilon$

$\leq \|\tilde{Y}_{(k-1)\tau,next^{u_k}(z_k)}^{\pi'}\| +(2k-1)\varepsilon$\ \  
for all $\pi'\in U^{k-1}$ by induction hypothesis,

$\leq \|\tilde{Y}_{(k-1)\tau,z'}^{\pi'}\| +2k\varepsilon$\ \  
for all $\pi'\in U^{k-1}$ and all $z'\in \{next^u(z)\ | u\in U\}$

$\leq \|\tilde{Y}_{\tau,z}^{\pi}\| +2k\varepsilon$\ \  
for all $\pi\in U^k$.

\hspace*{\fill} $\Box$
\end{proof}
\begin{theorem}\label{th:2}
Let $y\in S$ be a point
 of $\varepsilon$-representative $z\in {\cal X}$. Let
$\pi_k^\varepsilon$ be the pattern returned by $PROC_k^\varepsilon(z)$,
and $\pi^* := \mbox{argmin}_{\pi\in U_k} \|Y^\pi_{k\tau,y}-y_{f}\|$.
The discretization error 
$E_\varepsilon(T) :=|\|\tilde{Y}^{\pi_k^\varepsilon}_{k\tau,z}-y_f\|-{\bf v}_k(y)|$, with ${\bf v}_k(y) := \|Y_{k\tau,y}^{\pi^*}-y_f\|$ and $T=k\tau$,
satisfies:
$$E_\varepsilon(T)\leq (2k+1)\varepsilon.$$
It follows that $\|\tilde{Y}_{k\tau,z}^{\pi_k^\varepsilon}-y_f\|$ converges to 
${\bf v}_k(y)$ as $\varepsilon\rightarrow 0$.
\end{theorem}
\begin{proof}
W.l.o.g., let us suppose that $y_f$ is the origin $O$.
For all $\pi\in U^k$, we have by 
Proposition \ref{prop:inf} and Theorem \ref{th:1}:
$$\|\tilde{Y}_{k\tau,z}^{\pi_k^\varepsilon}\|\leq \|\tilde{Y}_{k\tau,z}^\pi\|+2k\varepsilon\leq \|Y_{k\tau,y}^\pi\|+(2k+1)\varepsilon.$$
Hence
$$\|\tilde{Y}_{k\tau,z}^{\pi_k^\varepsilon}\|\leq \min_{\pi\in U^k}
\|Y_{k\tau,y}^\pi\|+(2k+1)\varepsilon=\|Y_{k\tau,y}^{\pi^*}\|+(2k+1)\varepsilon.$$
On the other hand, for all $\pi\in U^k$, it follows from Theorem \ref{th:1}:
$$\|Y_{k\tau,y}^{\pi^*}\|\leq \|Y_{k\tau,y}^\pi\|\leq \|\tilde{Y}_{k\tau,z}^\pi\|+\varepsilon.$$
Hence:
$$\|Y_{k\tau,y}^{\pi^*}\|\leq \|\tilde{Y}_{k\tau,z}^{\pi_k^{\varepsilon}}\|+\varepsilon.$$
Therefore we have: $|\|\tilde{Y}_{k\tau,z}^{\pi_k^{\varepsilon}}\|-\|Y_{k\tau,y}^{\pi^*}\||\leq (2k+1)\varepsilon$.
\hspace*{\fill} $\Box$
\end{proof}

\begin{remark}\label{rk:pattern}
The error bound $E_\varepsilon(T)$ 
is thus {\em linear} in $k=T/\tau$. 
In order to decrease~$k$, one can apply 
consecutively $p\geq 2$ modes {\em in a row} (without intermediate
$\varepsilon$-approximation); this is equivalent to divide $k$ by $p$, 
at the price of considering $m^p$ ``extended'' modes instead of
just $m$ modes. 
(see Example~\ref{ex:1}, Figure~\ref{fig:Pouchol52}).
An alternative for decreasing $k$ is 
to increase $\tau$ (which may require in turn to
decrease~$\Delta t$ for preserving assumption
$\Delta t \leq G_u(1-\alpha_u)$, see Remark~\ref{rk:subsampling}).
%
\end{remark}
\begin{example} \label{ex:1}
Consider the 1D reaction-diffusion system with Dirichlet boundary
condition (see \cite{Pouchol18}, bistable case): 
\begin{align*}
& \frac{\partial y(t,x)}{\partial t} =\sigma \frac{\partial^2 y(t,x)}{\partial x^2} + f(y(t,x)),\ \ \ t\in[0,T],\ x\in [0,L]
\\
& y(t,0) = u_0, \quad y(t,L) = u_L, 
\\
& y(0,x) = y_0(x),\ \ \ x\in [0,L]
\end{align*}
with $\sigma=1, L=4$ and $f(y) = y(1-y)(y-\theta)$ with $\theta = 0.3$.
The control switching period is $\tau = 0.1$. 
The values of the boundary control  $u=(u_0, u_L)$ are in
$$U=\{(0,0),(0.2,0.2),(0.4,0.4),(0.6,0.6),(0.8.0.8),(1,1)\}.\footnote{Note that, in \cite{Pouchol18}, the values of the boundary control are in
the full interval $[0,1]$, not in a finite set $U$ as here. In \cite{Pouchol18}, they focus, not on the bounding of
computation errors during integration as here, but on a formal  
proof that 
the objective state
$y_f=\theta$
($0<\theta<1$) is reachable in finite time iff $L<L^*$ for some threshold value~$L^*$.}$$
We discretize the domain $\Omega=[0,L]$ 
of the system with $M_1=5$ discrete points, using a finite difference scheme. 
Our program returns an OSL constant
$\lambda_u = -0.322$ for all $u\in U$.
Constant 
$C_u  $ varies between $10.33$ and $11.85$ depending on the values of $u$.

We then discretize each interval component
of the space $S=[0,1]^{M_1}$ of values of $y$
into 15 points with spacing 
$\eta=1/15\approx 0.066$. The grid~${\cal X}$ is of the form
$\{0,\eta, 2\eta,\dots, 15\eta\}^{M_1}$, and
the initial error $e_0$ equal to $\varepsilon=\sqrt{M_1}\eta/2$.
This leads to $G_u$ varying between $0.00155$ and $0.00178$ depending 
on the value of $u\in U$. One checks: $\frac{|\lambda_u|G_u}{4}<1$
for all $u\in U$.
The time step upper bound required by Theorem \ref{th:1} for
ensuring numeric stability is $0.00155$. 
Since the switching period is $\tau = 0.1$, 
we perform  {\em subsampling}
(see, e.g., \cite{SNR17}) by decomposing every time step $[i\tau,(i+1)\tau)$
($1\leq i\leq k-1$)
into a sequence of elementary Euler steps of length 
$\Delta t = \tau / 100< 0.00155$.
This ensures that the system satisfies $(H)$, hence, by Theorem~\ref{th:1},
the explicit Euler scheme is stable and error 
$\|Y^\pi_{t,y_0}-\tilde{Y}_{t,z_0}^\pi\|$ never exceeds $\varepsilon$.  

For objective with $y_f=(0.3,0.3,0.3,0.3,0.3)$
 and horizon time $T=k\tau=2$ (i.e., $k=20$),
our program\footnote{The program, called ``OSLator'' \cite{oslator}, is implemented in Octave. It is composed of 10 functions and a main script totalling 600 lines of code. The computations are realised in a virtual machine running Ubuntu 18.06 LTS, having access to one core of a 2.3GHz Intel Core i5, associated to 3.5 GB of RAM memory.} returns an approximate optimal controller in $2$ minutes.  
Let $z_0$ be the $\varepsilon$-representative of
$y_0= 0.8 x/L + 0.1 (1 - x/L)$.
Let $\pi_k^\varepsilon$ be the pattern output by $PROC_k^\varepsilon(z_0)$.
A simulation of $z(t) :=\tilde{Y}_{t,z_0}^{\pi_k^\varepsilon}$
is given in Figure~\ref{fig:Pouchol5}
with $T=2$, $\tau=0.1$ ($k=20$),
$\Delta t=\frac{\tau}{100}$. We have
$\| z(T)-yf \|\approx 0.276$.
The simulation presents some similarity with
simulations displayed in \cite{Pouchol18}
(see, e.g., lower part of Figure 6), with a phase
control $u_0=u_L>\theta$ (here, $u_0=u_L=0.4$) alternating with a phase
control $u_0=u_L<\theta$ (here, $u_0=u_L=0.2$).
The discretization error $E_\varepsilon(T)$
is smaller than $(2k+1)\varepsilon=41\sqrt{5}/30< 3.1$.

\begin{figure}[t]
  \centering
    \includegraphics[scale=0.4]{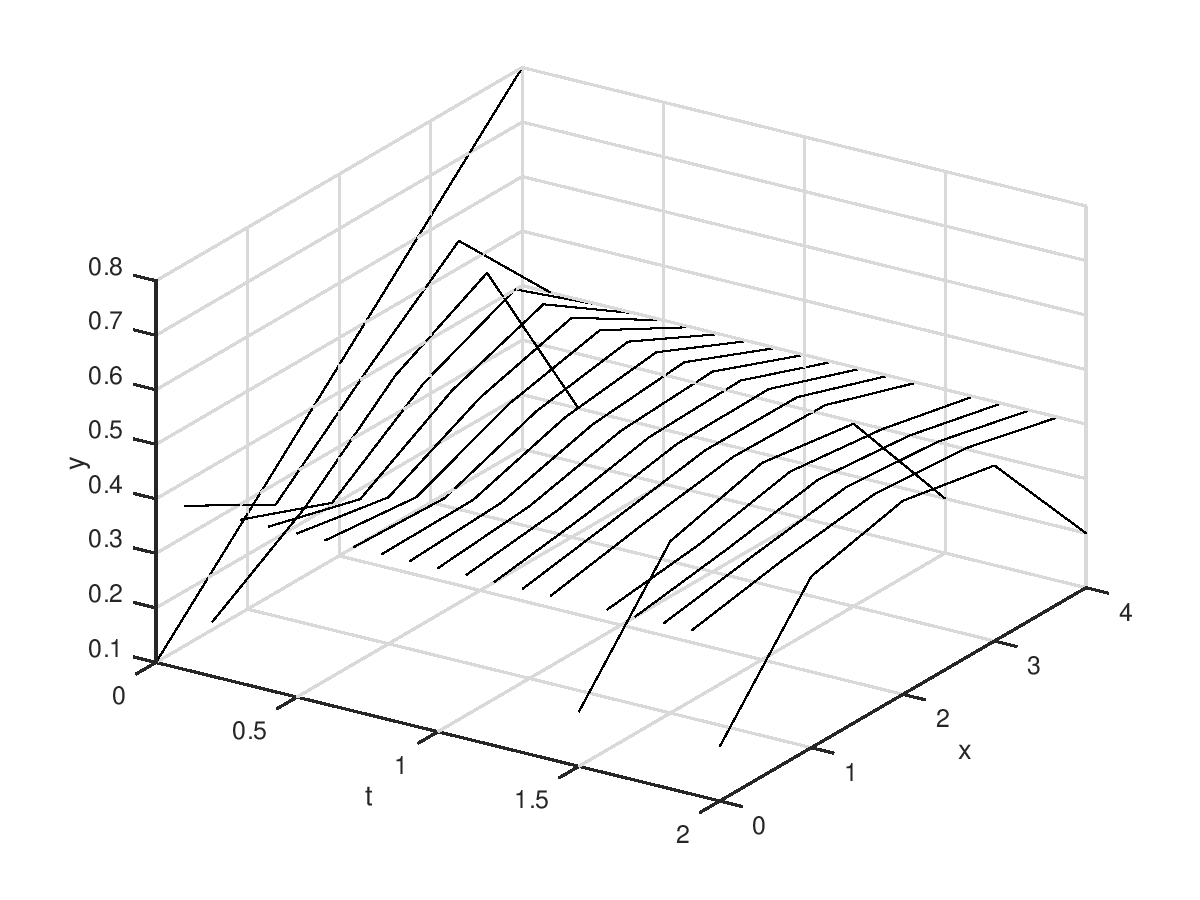}
  \caption{Simulation of the system of Example~\ref{ex:1} discretized
with $M_1=5$ points, for initial condition 
$y_0= 0.8 x/L + 0.1 (1 - x/L) $, objective
$y_f=0.3$ and horizon time $T=2$ ($\tau=0.1$, $\Delta t=\frac{\tau}{100}$).}
  \label{fig:Pouchol5}
\end{figure}
%

Let us now proceed with extended modes of length $p=2$ and 
$p=4$, as explained
in Remark~\ref{rk:pattern}.
For $p=2$ (i.e., $k=10$),
the control is synthesized in 7mn of CPU time.
The controller simulation is given in 
the left part of Figure~\ref{fig:Pouchol52};
we have:
$\| z(T)-y_f \|\approx 0.445$ 
with $E_\varepsilon(T) < 1.57$.
For $p=4$ (i.e., $k=5$), the computation of the control requires
8h of CPU time.
The corresponding simulation 
is given in the right part of Figure~\ref{fig:Pouchol52};
we now have:
$\|z(T)-y_f \|\approx 0.164$
with $E_\varepsilon(T)< 0.82$.
%

\end{example}

\subsection{Model reduction}\label{ss:MOR}
Let us consider the system ${\cal S}_2$
on space $S_{h_2}=[0,1]^{M_2}$ (with $M_2$ even).
The differential equation
can be written under the form:
$$\frac{dy_2}{dt}=\sigma{\cal L}_{h_2}y_2+\varphi_{h_2}(u)+{f}(y_2).$$
where ${\cal L}_{h_2}$
corresponds to the $(M_2\times M_2)$ Laplacian matrix,
and $h_2=\frac{L}{M_2+1}$.

Let us consider the ``reduced'' system~${\cal S}_1$ defined 
on $S_{h_1}=[0,1]^{M_1}$ with $M_1= M_2/2$, defined by:
$$\frac{dy_1}{dt}=\sigma{\cal L}_{h_1}y_1+\varphi_{h_1}(u)+f({y}_1),$$
%
where ${\cal L}_{h_1}$ is the $(M_1\times M_1)$ Laplacian matrix
and $h_1=\frac{L}{M_1+1}$.

With $M_1 = M_2/2$, we have $h_2 = \frac{L}{2M_1 +1}$ ($ =\frac{h_1(M_1+1)}{2M_1+1}$).
Let us consider the $(M_1\times M_2)$ {\em reduction matrix}:
\[
\Pi :=\frac{1}{\sqrt{2}}
  \begin{bmatrix}
    1 & 1 & 0 & \cdots & 0 & 0\\
    0 & 0 & 1 & 1 &\cdots & 0\\
    \ & \ & \cdots &\ &\ &\ \\
    0 & 0 & \cdots & 0 & 1 & 1
  \end{bmatrix}
\]
Note that $\Pi\Pi^\top={\cal I}_{M}$. 
%
Let us consider a point $w_0\in S_{h_2}$,
and let $z_0=\Pi w_0\in S_{h_1}$.
%
%
%
%
%

\begin{theorem}
\label{th:3}
Consider the system ${\cal S}_2$ and a point $w_0\in S_{h_2}$,
and let $z_0=\Pi w_0\in S_{h_1}$.
Let $Y_{w_0}^{h_2}$ and $Y_{z_0}^{h_1}$
be the solutions of ${\cal S}_2$ and ${\cal S}_1$ with 
initial conditions $w_0\in S_{h_2}$
and $z_0\in S_{h_1}$ respectively.
We have:
$$\forall t\geq 0\ \ \ \|\Pi Y_{w_0}^{h_2}(t)-Y_{z_0}^{h_1}(t)\|\leq \frac{K_2\sigma}{|\lambda_{h_1}|},$$
where 
$$K_2 := \sup_{w\in S_{h_2}}\|(\Pi  {\cal L}_{h_2}  -{\cal L}_{h_1}\Pi )w\|,$$
and ${\cal L}_{h_2}$ (resp. ${\cal L}_{h_1}$)
is the Laplacian matrix of size $M_2\times M_2$ (resp. $M_1\times M_1$).
\end{theorem}
\begin{proof}
Let us consider the system ${\cal S}_2$:
$$\frac{dy_2}{dt}=\sigma{\cal L}_{h_2}y_2+\varphi_{h_2}(u)+f(y_2).$$
%
By application of the projection matrix $\Pi $, we get:
$$\frac{d\Pi  y_2}{dt}=\sigma\Pi  {\cal L}_{h_2}y_2+\varphi_{{h_1}}(u)+f(\Pi  y_2).$$
By substracting pairwise with the sides of ${\cal S}_1$,
we have:
$$\frac{d \Pi  y_2}{dt}-\frac{dy_1}{dt}=\sigma (\Pi  {\cal L}_{h_2} y_2
-{\cal L}_{{h_1}}y_1) +f(\Pi  y_2)-f(y_1)$$
$$=F_{h_1}(\Pi  y_2)-F_{h_1}(y_1)+\sigma(\Pi {\cal L}_{h_2}-{\cal L}_{h_1}\Pi ) y_2,$$
where $F_{h_1}(y)=\sigma{\cal L}_{h_1}(y)+f(y)$ for $y\in S_{h_1}$.
On the other hand, we have: 

$\frac{1}{2}\frac{d}{dt}(\|\Pi  y_2-y_1\|^2) =\langle \frac{d}{dt}(\Pi  y_2-y_1),\Pi  y_2-y_1\rangle$

$=\langle F_{h_1}(\Pi  y_2)-F_{h_1}(y_1)+\sigma(\Pi {\cal L}_{h_2}-{\cal L}_{h_1}\Pi ) y_2,
\Pi  y_2-y_1\rangle$

$=\langle F_{h_1}(\Pi  y_2)-F_{h_1}(y_1), \Pi  y_2-y_1\rangle$

\hspace*{8mm} $+\sigma \langle (\Pi {\cal L}_{h_2}-{\cal L}_{h_1}\Pi ) y_2,\Pi  y_2-y_1\rangle$

$\leq \lambda_{h_1} \|\Pi  y_2-y_1\|^2
+\sigma \langle (\Pi {\cal L}_{h_2}-{\cal L}_{h_1}\Pi ) y_2,\Pi  y_2-y_1\rangle$

$\leq \lambda_{h_1} \|\Pi  y_2-y_1\|^2
+K_2\sigma \|\Pi  y_2-y_1\|$
$$\mbox{with}\ \ \ \ K_2 := \sup_{w\in S_{h_2}}\|(\Pi  {\cal L}_{h_2}  -{\cal L}_{{h_1}}\Pi )w\|$$
\hspace*{3mm} $\leq \lambda_{h_1} \|\Pi  y_2-y_1\|^2
+K_2\sigma \frac{1}{2}(\alpha\|\Pi  y_2-y_1\|^2+\frac{1}{\alpha})$,\\
for all $\alpha>0$. Choosing $\alpha>0$ such that $K_2\sigma\alpha=-\lambda_{h_1}$,
i.e.: $\alpha=-\frac{\lambda_{h_1}}{K_2\sigma}$, we have:
$$\frac{1}{2} \frac{d}{dt}(\|\Pi  y_2-y_1\|^2)
\leq \frac{\lambda_{h_1}}{2}\|\Pi  y_2-y_1\|^2-\frac{(K_2\sigma)^2}{2\lambda_{h_1}}.$$
Since $y_2(0)=w_0$ and $y_1(0)=z_0$, we get by integration:
$$\|\Pi  y_2(t) -y_1(t)\|^2 \leq \frac{(K_2\sigma)^2}{\lambda_{h_1}^2}(1-e^{\lambda_{h_1}t})\leq \frac{(K_2\sigma)^2}{\lambda_{h_1}^2}.$$
Hence:
$\|\Pi  Y^{h_2}_{w_0}(t)-Y_{z_0}^{h_1}(t)\|\leq \frac{K_2\sigma}{|\lambda_{{h_1}}|}$
for all $t\geq 0$.
\hfill $\Box$
\end{proof}
This proposition expresses that 
the reduction error 
is bounded by  constant $\frac{K_2\sigma}{|\lambda_{{h_1}}|}$  when the same control modes are applied to both systems.\footnote{By comparison, in \cite{AllaS19}, the 
error term originating from
the POD model reduction is {\em exponential} in $T$
(see $C_1(T,|x|)$ in the proof of Theorem 5.1).}

%
Let $y_2^0\in S_2$ and $y_2^f\in S_2$
be an initial and  objective point respectively.
Let $y_1^0:= \Pi y_2^0\in S_1$ and $y_1^f:= \Pi y_2^f\in S_1$ denote their projections.
Suppose that $\pi^\varepsilon$ is  the pattern returned 
by $PROC^\varepsilon_k(y_1^0)$ for the reduced system
${\cal S}_1$.
%
Then, from Theorem~\ref{th:3}, it follows that,
when the {\em same} control $\pi^\varepsilon$ 
is applied to the original system ${\cal S}_2$ with 
$y_2(0)=y_2^0\in S_2$, it makes the projection $\Pi y_2^{\pi^\varepsilon}(t)\in S_1$
reach a {\em neighborhood} of $y_1^f$ at time $t=T$.
Formally, we have:
$$\|P y_2^{\pi^\varepsilon}(T)- y_1^f\|\leq \|y_1^{\pi^\varepsilon}(T)-y_1^f\|+ \frac{K_2\sigma}{|\lambda_{{h_1}}|}.$$
\begin{example}\label{ex:2}
Let us take
the system defined in Example \ref{ex:1} as reduced system ${\cal S}_1$ 
($M_1=5$), and let us take as ``full-size'' system ${\cal S}_2$
the system corresponding to
$M_2=10$. Since the size of the grid  ${\cal X}_2$ associated to ${\cal S}_2$
is exponential in $M_2$, the
size ${\cal X}_2$ is multiplied by 
$(1/\eta)^{M_2-M_1}=15^5\approx 7.6\cdot 10^{5}$ w.r.t. the size
of the grid ${\cal X}_1$ associated to ${\cal S}_1$. The complexity
for synthesizing directly 
the optimal control of ${\cal S}_2$ thus becomes intractable.
On the other hand, if we apply to ${\cal S}_2$
the optimal strategy $\pi^\varepsilon\in U^k$ found
for ${\cal S}_1$ in Example~\ref{ex:1}, we obtain 
a simulation depicted in  Figure~\ref{fig:Pouchol10.0} for extended mode
of length 1, which is the counterpart of 
Figure~\ref{fig:Pouchol5}
with $M_2=10$ (instead of $M_1=5$), and has a very similar form. 
Likewise, if we apply to ${\cal S}_2$
the optimal strategy $\pi^\varepsilon\in U^k$ found
for ${\cal S}_1$ in Example~\ref{ex:1}, we obtain 
a simulation depicted in  Figure~\ref{fig:Pouchol10} for extended modes
of length 2 and 4, which is the counterpart of
Figure~\ref{fig:Pouchol52}, and very similar to it.
As seen above, we have:
$$\| \Pi y_2^{\pi^\varepsilon}(T)- y_1^f\|\leq \|y_1^{\pi^\varepsilon}(T)-y_1^f\|+ \frac{K_2\sigma}{|\lambda_{{h_1}}|},$$ 
where $y_1^f=(0.3,0.3,0.3,0.3,0.3)$, and
the reduction error is bounded by $\frac{K_2\sigma}{|\lambda_{{h_1}}|} = 17.9\ \sigma.$

%
The subexpression $\|y_1^{\pi^\varepsilon}(T)-y_1^f\|$ can be computed 
{\em a posteriori} by simulation: see Table~1 of Appendix~2,
with $\sigma=1$, $\sigma=0.5$.
The value of 
$\|y_2^{\pi^\varepsilon}(T)-y_2^f\|$ for~${\cal S}_2$ is also given
in Table 1 for comparison.

The upper bound $\|y_1^{\pi^\varepsilon}(T)-y_1^f\|+ \frac{K_2\sigma}{|\lambda_{{h_1}}|}$
of the distance
$\|Py_2^{\pi^\varepsilon}(T)-y_1^f\|$
is very conservative, due to {\em a priori} error bound  $\frac{K_2\sigma}{|\lambda_{{h_1}}|}$.
On can obtain {\em a posteriori}
a much sharper estimate of $\|Py^{\pi^\varepsilon}_2(T)-y_1^f\|$ 
by simulation: see Table~2, Appendix~2.
\begin{figure}[t]
  \includegraphics[scale=0.32]{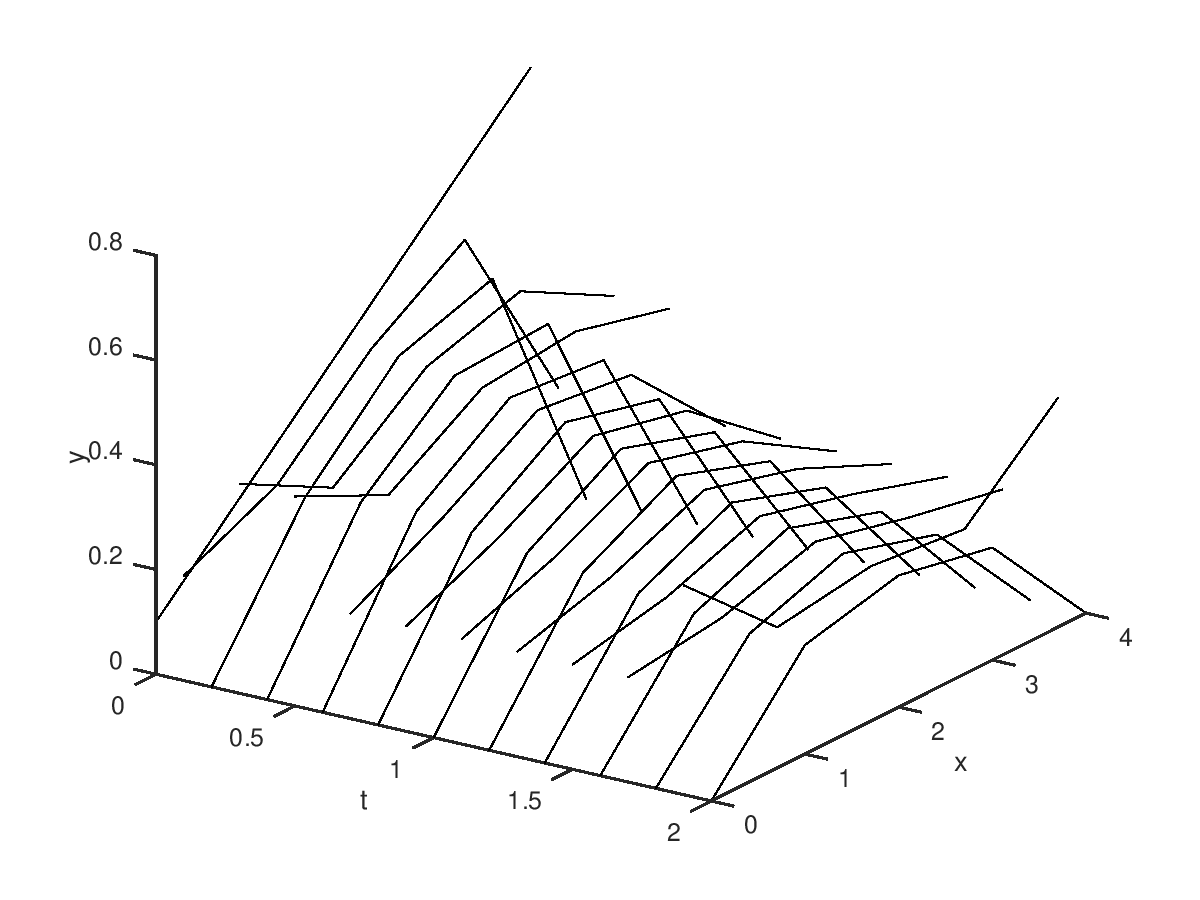}\includegraphics[scale=0.32]{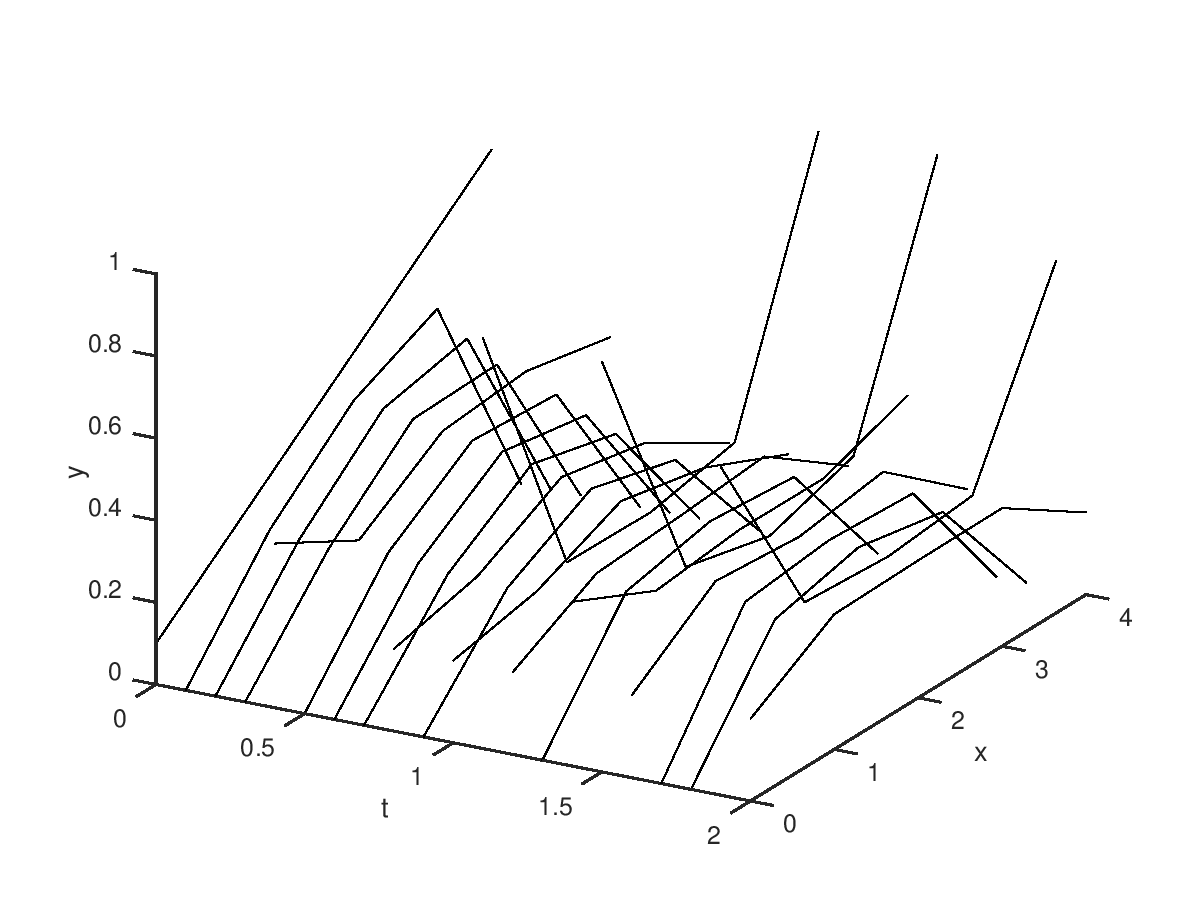}
  \caption{Simulation of the system of Example~\ref{ex:1} 
discretized with $M_1=5$ points,
with extended modes of length 2 (left) and extended modes
of length 4 (right).
}
  \label{fig:Pouchol52}
\end{figure}

\begin{figure}[t]
  \centering
    \includegraphics[scale=0.4]{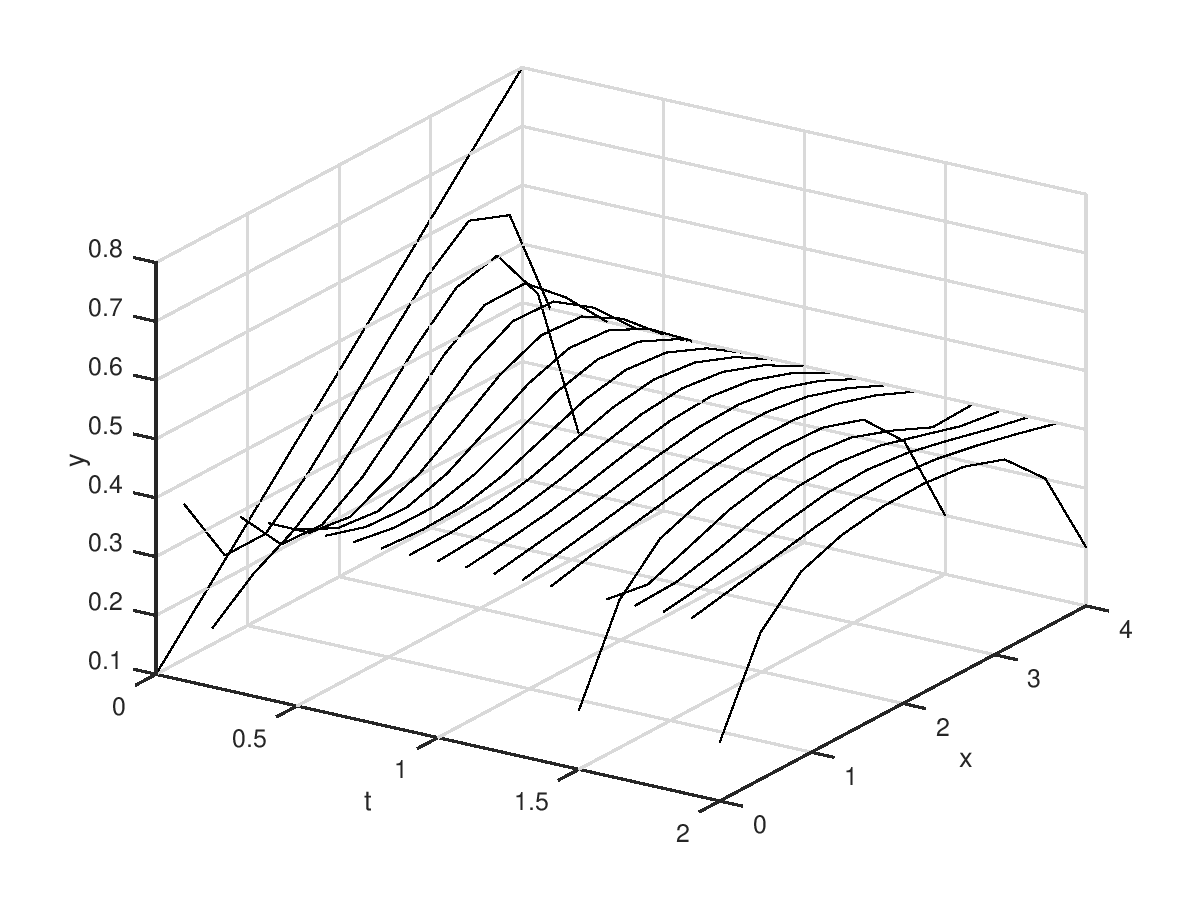}
  \caption{Simulation of the  system of Ex. \ref{ex:1},
discretized with $M_2=10$ points,
with extended mode of length 1.
}
  \label{fig:Pouchol10.0}
\end{figure}

\begin{figure}[t]
  \centering
    \includegraphics[scale=0.3]{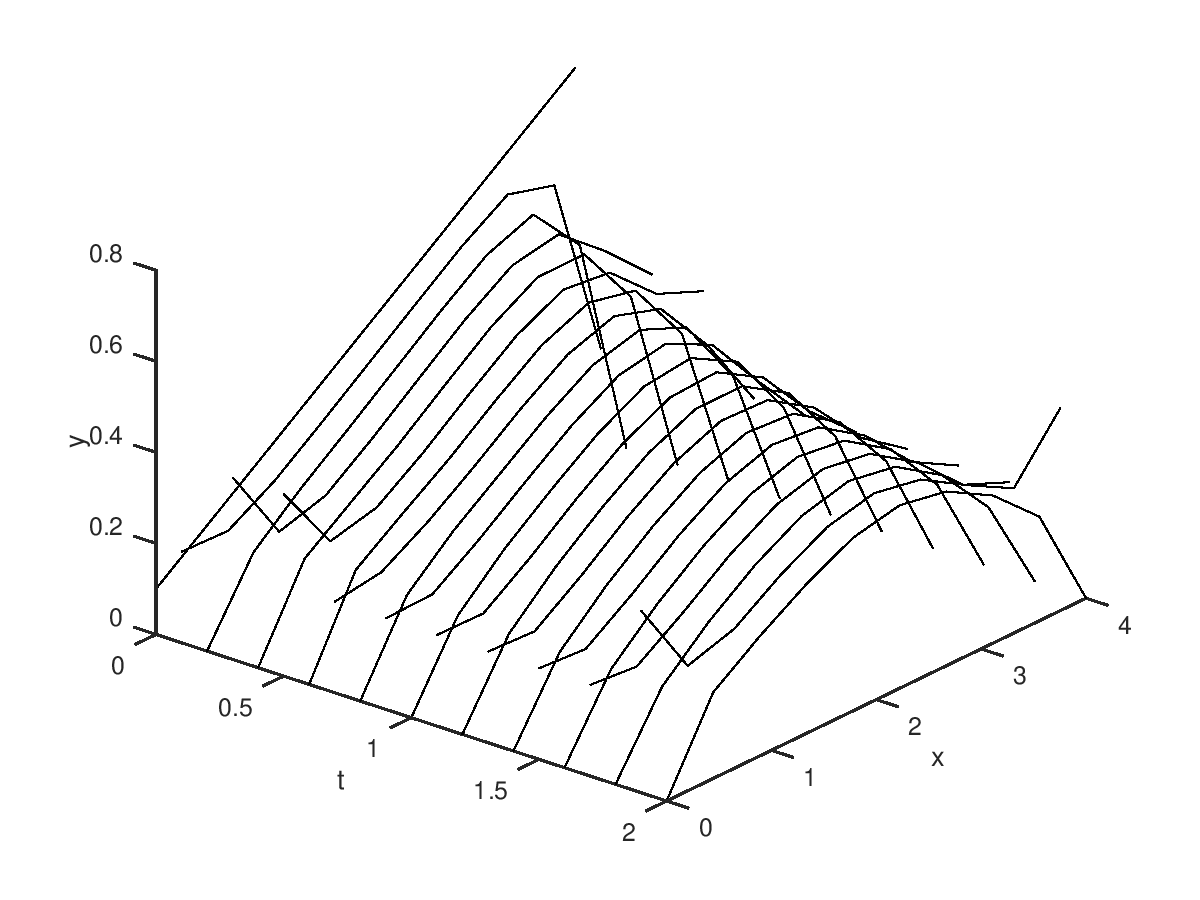}\includegraphics[scale=0.3]{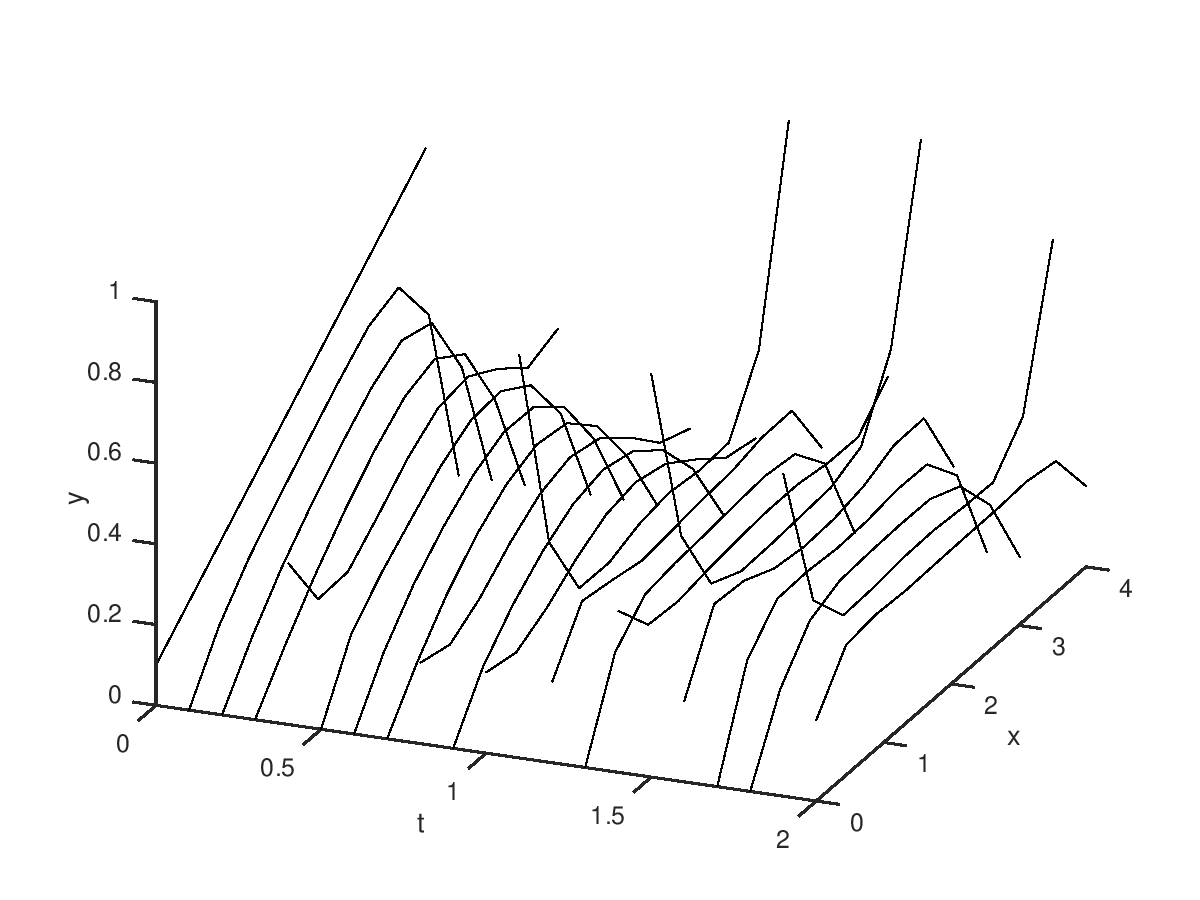}
  \caption{Simulation of the  system of Ex. \ref{ex:1},
discretized with $M_2=10$ points,
with extended modes of length 2 (left) and 
extended modes of length 4 (right).
}
  \label{fig:Pouchol10}
\end{figure}

\end{example}

\section{Final Remarks}\label{sec:final}

Using the notion of OSL constant, we have shown how to use
the finite difference and explicit Euler methods in order
to solve finite horizon control problems
for reaction-diffusion equations. Furthermore, we have 
quantified the deviation of this control with the optimal strategy,
and proved that the error upper bound is {\em linear} in the 
horizon length.
We have applied the method
to a 1D bi-stable reaction-diffusion equation, and
have found experimental results similar to those of
\cite{Pouchol18}. We have also given a simple and specific 
model reduction
method which allows to apply the method to equations of larger size.
In future work, we plan to apply the method to 2D reaction-diffusion equations
(e.g., Test 1 of \cite{AllaS19}).









%

\bibliographystyle{plain}
\bibliography{euler}

\begin{thebibliography}{10}

\bibitem{AllaFV17}
Alessandro Alla, Maurizio Falcone, and Stefan Volkwein.
\newblock Error analysis for {POD} approximations of infinite horizon problems
  via the dynamic programming approach.
\newblock {\em {SIAM} J. Control and Optimization}, 55(5):3091--3115, 2017.

\bibitem{AllaS19}
Alessandro Alla and Luca Saluzzi.
\newblock A {HJB-POD} approach for the control of nonlinear {PDE}s on a tree
  structure.
\newblock {\em CoRR}, abs/1905.03395, 2019.

\bibitem{Althoff17}
Matthias Althoff.
\newblock Reachability analysis of large linear systems with uncertain inputs
  in the {K}rylov subspace.
\newblock {\em CoRR}, abs/1712.00369, 2017.

\bibitem{AlthoffSB08}
Matthias Althoff, Olaf Stursberg, and Martin Buss.
\newblock Reachability analysis of nonlinear systems with uncertain parameters
  using conservative linearization.
\newblock In {\em Proceedings of the 47th {IEEE} Conference on Decision and
  Control, {CDC} 2008, December 9-11, 2008, Canc{\'{u}}n, Mexico}, pages
  4042--4048. {IEEE}, 2008.

\bibitem{aminzare2014guaranteeing}
Zahra Aminzare, Yusef Shafi, Murat Arcak, and Eduardo~D Sontag.
\newblock Guaranteeing spatial uniformity in reaction-diffusion systems using
  weighted $l^{2}$ norm contractions.
\newblock In {\em A Systems Theoretic Approach to Systems and Synthetic Biology
  I: Models and System Characterizations}, pages 73--101. Springer, 2014.

\bibitem{aminzare2013logarithmic}
Zahra Aminzare and Eduardo~D Sontag.
\newblock Logarithmic lipschitz norms and diffusion-induced instability.
\newblock {\em Nonlinear Analysis: Theory, Methods \& Applications}, 83:31--49,
  2013.

\bibitem{aminzare2016some}
Zahra Aminzare and Eduardo~D Sontag.
\newblock Some remarks on spatial uniformity of solutions of
  reaction--diffusion pdes.
\newblock {\em Nonlinear Analysis: Theory, Methods \& Applications},
  147:125--144, 2016.

\bibitem{arcak2011certifying}
Murat Arcak.
\newblock Certifying spatially uniform behavior in reaction--diffusion pde and
  compartmental ode systems.
\newblock {\em Automatica}, 47(6):1219--1229, 2011.

\bibitem{Barthel2010}
Werner Barthel, Christian John, and Fredi Tr\"{o}ltzsch.
\newblock Optimal boundary control of a system of reaction diffusion equations.
\newblock {\em ZAMM - Journal of Applied Mathematics and Mechanics /
  Zeitschrift f\"{u}r Angewandte Mathematik und Mechanik}, 90(12):966--982,
  2010.

\bibitem{Bellman57}
Richard Bellman.
\newblock {\em Dynamic Programming}.
\newblock Princeton University Press, Princeton, NJ, USA, 1 edition, 1957.

\bibitem{BerzH98}
Martin Berz and Georg Hoffst{\"{a}}tter.
\newblock Computation and application of taylor polynomials with interval
  remainder bounds.
\newblock {\em Reliable Computing}, 4(1):83--97, 1998.

\bibitem{BerzM98}
Martin Berz and Kyoko Makino.
\newblock Verified integration of {ODE}s and flows using differential algebraic
  methods on high-order {T}aylor models.
\newblock {\em Reliable Computing}, 4(4):361--369, 1998.

\bibitem{CasasRT18}
Eduardo Casas, Christopher Ryll, and Fredi Tr{\"{o}}ltzsch.
\newblock Optimal control of a class of reaction-diffusion systems.
\newblock {\em Comp. Opt. and Appl.}, 70(3):677--707, 2018.

\bibitem{LeCoentF19}
Adrien~Le Co{\"{e}}nt and Laurent Fribourg.
\newblock Guaranteed control of sampled switched systems using
  semi-{L}agrangian schemes and one-sided {L}ipschitz constants.
\newblock In {\em 58th {IEEE} Conference on Decision and Control, {CDC} 2019,
  Nice, France, December 11-13}, 2019.

\bibitem{CourtKP18}
S\'ebastien Court, Karl Kunisch, and Laurent Pfeiffer.
\newblock Hybrid optimal control problems for a class of semilinear parabolic
  equations.
\newblock {\em Discrete \& Continuous Dynamical Systems - S}, 11, 2018.

\bibitem{Estrela17}
Jorge~Estrela da~Silva, Joao~Tasso Sousa, and Fernando~Lobo Pereira.
\newblock Synthesis of safe controllers for nonlinear systems using dynamic
  programming techniques.
\newblock In {\em 8th International Conference on Physics and Control (PhysCon
  2017)}. {IPACS} {E}lectronic library, 2017.

\bibitem{dahlquist1958stability}
Germund Dahlquist.
\newblock {\em Stability and error bounds in the numerical integration of
  ordinary differential equations}.
\newblock PhD thesis, Almqvist \& Wiksell, 1958.

\bibitem{Falcone1999}
Maurizio Falcone and Tiziana Giorgi.
\newblock An approximation scheme for evolutive hamilton-jacobi equations.
\newblock In {\em Stochastic analysis, control, optimization and applications},
  pages 289--303. Springer, 1999.

\bibitem{fan2018simulation}
Chuchu Fan, James Kapinski, Xiaoqing Jin, and Sayan Mitra.
\newblock Simulation-driven reachability using matrix measures.
\newblock {\em ACM Transactions on Embedded Computing Systems (TECS)},
  17(1):21, 2018.

\bibitem{Finotti12}
Heather Finotti, Suzanne Lenhart, and Tuoc~Van Phan.
\newblock Optimal control of advective direction in reaction-diffusion
  population models.
\newblock {\em Evolution Equations \& Control Theory}, 1, 2012.

\bibitem{girard2005reachability}
Antoine Girard.
\newblock Reachability of uncertain linear systems using zonotopes.
\newblock In {\em Proc. of Hybrid Systems: Computation and Control}, volume
  3414 of {\em LNCS}, pages 291--305. Springer, 2005.

\bibitem{GriesseV05}
Roland Griesse and Stefan Volkwein.
\newblock A primal-dual active set strategy for optimal boundary control of a
  nonlinear reaction-diffusion system.
\newblock {\em {SIAM} J. Control and Optimization}, 44(2):467--494, 2005.

\bibitem{HanK04}
Zhi Han and Bruce~H. Krogh.
\newblock Reachability analysis of hybrid control systems using reduced-order
  models.
\newblock In {\em Proceedings of the 2004 American Control Conference},
  volume~2, pages 1183--1189 vol.2, June 2004.

\bibitem{HanK06}
Zhi Han and Bruce~H. Krogh.
\newblock Reachability analysis of large-scale affine systems using
  low-dimensional polytopes.
\newblock In {\em Hybrid Systems: Computation and Control, 9th International
  Workshop, {HSCC} 2006, Santa Barbara, CA, USA, March 29-31, 2006,
  Proceedings}, pages 287--301, 2006.

\bibitem{kalisek14}
Dante Kalise and Axel Kr{\"o}ner.
\newblock {Reduced-order minimum time control of advection-reaction-diffusion
  systems via dynamic programming.}
\newblock In {\em {21st International Symposium on Mathematical Theory of
  Networks and Systems}}, pages 1196--1202, Groningen, Netherlands, July 2014.

\bibitem{KaliseK18}
Dante Kalise and Karl Kunisch.
\newblock Polynomial approximation of high-dimensional hamilton-jacobi-bellman
  equations and applications to feedback control of semilinear parabolic pdes.
\newblock {\em {SIAM} J. Scientific Computing}, 40(2), 2018.

\bibitem{kapela2009lohner}
Tomasz Kapela and Piotr Zgliczy{\'n}ski.
\newblock A lohner-type algorithm for control systems and ordinary differential
  inclusions.
\newblock {\em Discrete \& Continuous Dynamical Systems-B}, 11(2):365--385,
  2009.

\bibitem{Koto08}
Toshiyuki Koto.
\newblock {IMEX} {R}unge-{K}utta schemes for reaction-diffusion equations.
\newblock {\em Journal of Computational and Applied Mathematics},
  215(1):182--195, 2008.

\bibitem{Kuhn:1998:RCO:287056.287083}
W.~K\"{u}hn.
\newblock Rigorously computed orbits of dynamical systems without the wrapping
  effect.
\newblock {\em Computing}, 61(1):47--67, 1998.

\bibitem{KunischVX04}
Karl Kunisch, Stefan Volkwein, and Lei Xie.
\newblock Hjb-pod-based feedback design for the optimal control of evolution
  problems.
\newblock {\em {SIAM} J. Applied Dynamical Systems}, 3(4):701--722, 2004.

\bibitem{oslator}
Adrien Le~Co\"ent.
\newblock {OSL}ator 1.0.
\newblock \url{https://bitbucket.org/alecoent/oslator/src/master/}, 2019.

\bibitem{AdrienRP17}
Adrien Le~Co\"ent, Julien Alexandre Dit~Sandretto, Alexandre Chapoutot, Laurent
  Fribourg, Florian De~Vuyst, and Ludovic Chamoin.
\newblock Distributed control synthesis using {E}uler's method.
\newblock In {\em Proc. of International Workshop on Reachability Problems
  (RP'17)}, volume 247 of {\em Lecture Notes in Computer Science}, pages
  118--131. Springer, 2017.

\bibitem{SNR17}
Adrien Le~Co\"ent, Florian De~Vuyst, Ludovic Chamoin, and Laurent Fribourg.
\newblock Control synthesis of nonlinear sampled switched systems using
  {E}uler's method.
\newblock In {\em Proc. of International Workshop on Symbolic and Numerical
  Methods for Reachability Analysis (SNR'17)}, volume 247 of {\em EPTCS}, pages
  18--33. Open Publishing Association, 2017.

\bibitem{Syncop15}
Adrien Le~Co\"ent, Florian De~Vuyst, Christian Rey, Ludovic Chamoin, and
  Laurent Fribourg.
\newblock Guaranteed control synthesis of switched control systems using model
  order reduction and state-space bisection.
\newblock In {\em Proc. of International Workshop on Synthesis of Complex
  Parameters (SYNCOP'15)}, volume~44 of {\em {OASICS}}, pages 33--47. Schloss
  Dagstuhl -- Leibniz-Zentrum f\"{u}r Informatik, 2015.

\bibitem{Lohner87}
Rudolf~J. Lohner.
\newblock Enclosing the solutions of ordinary initial and boundary value
  problems.
\newblock {\em Computer Arithmetic}, pages 255--286, 1987.

\bibitem{lozinskii1958error}
Sergei~Mikhailovich Lozinskii.
\newblock Error estimate for numerical integration of ordinary differential
  equations. i.
\newblock {\em Izvestiya Vysshikh Uchebnykh Zavedenii. Matematika}, (5):52--90,
  1958.

\bibitem{maidens2014reachability}
John Maidens and Murat Arcak.
\newblock Reachability analysis of nonlinear systems using matrix measures.
\newblock {\em IEEE Transactions on Automatic Control}, 60(1):265--270, 2014.

\bibitem{MitchellBT01}
Ian~M. Mitchell, Alexandre~M. Bayen, and Claire~J. Tomlin.
\newblock Validating a hamilton-jacobi approximation to hybrid system reachable
  sets.
\newblock In {\em Hybrid Systems: Computation and Control, 4th International
  Workshop, {HSCC} 2001, Rome, Italy, March 28-30, 2001, Proceedings}, pages
  418--432, 2001.

\bibitem{MitchellT03}
Ian~M. Mitchell and Claire Tomlin.
\newblock Overapproximating reachable sets by hamilton-jacobi projections.
\newblock {\em J. Sci. Comput.}, 19(1-3):323--346, 2003.

\bibitem{Moore66}
Ramon Moore.
\newblock {\em Interval Analysis}.
\newblock Prentice Hall, 1966.

\bibitem{MouraF11}
Scott~J. {Moura} and Hosam~K. {Fathy}.
\newblock Optimal boundary control \& estimation of diffusion-reaction pdes.
\newblock In {\em Proceedings of the 2011 American Control Conference}, pages
  921--928, June 2011.

\bibitem{MouraF13}
{Scott J.} Moura and {Hosam K.} Fathy.
\newblock Optimal boundary control of reaction-diffusion partial differential
  equations via weak variations.
\newblock {\em Journal of Dynamic Systems, Measurement and Control,
  Transactions of the ASME}, 135(3), 6 2013.

\bibitem{Nedialkov99}
Nedialko~S. Nedialkov, K.~Jackson, and Georges Corliss.
\newblock Validated solutions of initial value problems for ordinary
  differential equations.
\newblock {\em Appl. Math. and Comp.}, 105(1):21 -- 68, 1999.

\bibitem{NedialkovKS04}
Nedialko~S. Nedialkov, Vladik Kreinovich, and Scott~A. Starks.
\newblock Interval arithmetic, affine arithmetic, taylor series methods: Why,
  what next?
\newblock {\em Numerical Algorithms}, 37(1-4):325--336, 2004.

\bibitem{Pouchol18}
Camille Pouchol, Emmanuel Tr\'{e}lat, and Enrique Zuazua.
\newblock Phase portrait control for {1D} monostable and bistable
  reaction-diffusion equations.
\newblock {\em CoRR}, abs/1709.07333, 2017.

\bibitem{ReissigR19}
Gunther Reissig and Matthias Rungger.
\newblock Symbolic optimal control.
\newblock {\em IEEE Transactions on Automatic Control}, 64(6):2224--2239, 2018.

\bibitem{RunggerR17}
Matthias Rungger and Gunther Reissig.
\newblock Arbitrarily precise abstractions for optimal controller synthesis.
\newblock In {\em 56th {IEEE} Annual Conference on Decision and Control, {CDC}
  2017, Melbourne, Australia, December 12-15, 2017}, pages 1761--1768, 2017.

\bibitem{Saluzzi18}
Luca Saluzzi, Alessandro Alla, and Maurizio Falcone.
\newblock Error estimates for a tree structure algorithm solving finite horizon
  control problems.
\newblock {\em CoRR}, abs/1812.11194, 2018.

\bibitem{SchurmannA17}
Bastian Sch{\"{u}}rmann and Matthias Althoff.
\newblock Optimal control of sets of solutions to formally guarantee
  constraints of disturbed linear systems.
\newblock In {\em 2017 American Control Conference, {ACC} 2017, Seattle, WA,
  USA, May 24-26, 2017}, pages 2522--2529, 2017.

\bibitem{SchurmannKA18}
Bastian Sch{\"{u}}rmann, Niklas Kochdumper, and Matthias Althoff.
\newblock Reachset model predictive control for disturbed nonlinear systems.
\newblock In {\em 57th {IEEE} Conference on Decision and Control, {CDC} 2018,
  Miami, FL, USA, December 17-19, 2018}, pages 3463--3470, 2018.

\bibitem{soderlind2006logarithmic}
Gustaf S{\"o}derlind.
\newblock The logarithmic norm. history and modern theory.
\newblock {\em BIT Numerical Mathematics}, 46(3):631--652, 2006.

\bibitem{sontag2010contractive}
Eduardo~D Sontag.
\newblock Contractive systems with inputs.
\newblock In {\em Perspectives in mathematical system theory, control, and
  signal processing}, pages 217--228. Springer, 2010.

\end{thebibliography}
\newpage
\section*{Appendix 1: Proof of Lemma \ref{lemma:1}}
\begin{proof}
It is easy to check that $0< \alpha_u< 1$ when $\frac{|\lambda_u|G_u}{4}<1$. 

Let $t^* := G_u(1-\alpha_u)$.
Let us first prove $\delta_{e_0}(t)\leq e_0$ for $t=t^*$.
We have:

$$-\frac{1}{2}|\lambda_{u}|  G_{u}
+(2+\frac{1}{2}|\lambda_{u}|G_{u})\alpha_u-\alpha_u^2 = 0.$$

Hence:

$$\frac{1}{2G_{u}(1-\alpha_u)}\lambda_{u}  G_{u}^2(1-\alpha_u)^2
+2\alpha_u-\alpha_u^2
= 0,$$

i.e.

$$\frac{1}{2t^*}\lambda_{u}  (t^*)^2
+2\alpha_u-\alpha_u^2
= 0.$$

We have:
$-\frac{1}{4G_u^2t^*}\lambda_{u}  (t^*)^4e^{\lambda_ut^*}\geq 0$. 
It follows:

$$\frac{1}{2t^*}\lambda_{u}  (t^*)^2
+2\alpha_u-\alpha_u^2-\frac{1}{4G_u^2t^*}\lambda_{u}  (t^*)^4
e^{\lambda_ut^*}\geq 0.$$

Hence:

$$1+\frac{1}{2t^*}\lambda_{u}  (t^*)^2
-\frac{1}{G_{u}^2}((t^*)^2+\frac{1}{4t^*}\lambda_{u}  (t^*)^4e^{\lambda_ut^*})
\geq 0.$$


By multiplying by $t^*$:

$$(t^*+\frac{1}{2}\lambda_{u}  (t^*)^2)
-\frac{1}{G_{u}^2}((t^*)^3+\frac{1}{4}\lambda_{u}  (t^*)^4e^{\lambda_ut^*})
\geq 0.$$

Since $G=\sqrt{3}|\lambda_u| e_0/C_u$:

$$e_0^2(t^*+\frac{1}{2}\lambda_{u}  (t^*)^2)
+\frac{C_u^2}{\lambda_u^2}(-\frac{1}{3}(t^*)^3-\frac{1}{12}\lambda_{u}  (t^*)^4e^{\lambda_ut^*})
\geq 0.$$

By multiplying by $\lambda_u$:

$$e_0^2(\lambda_ut^*+\frac{1}{2}\lambda_{u}^2  (t^*)^2)
+\frac{C_u^2}{\lambda_u^2}(-\frac{1}{3}\lambda_u(t^*)^3-\frac{1}{12}\lambda_{u}^2  (t^*)^4e^{\lambda_ut^*})
\leq 0.$$

Note that, in the above formula, the subexpression 
$\lambda_ut^*+\frac{1}{2}\lambda_{u}^2  (t^*)^2$ is such that:

$$\lambda_ut^*+\frac{1}{2}\lambda_{u}^2  (t^*)^2\geq e^{\lambda_u t^*}-1$$

since $e^{\lambda_u t^*}-1=\lambda_ut^*+\frac{1}{2}\lambda_{u}^2  (t^*)^2e^{\lambda\theta}\leq \lambda_ut^*+\frac{1}{2}\lambda_{u}^2  (t^*)^2$.\\

On the other hand, the subexpression
$-\frac{1}{3}\lambda_u(t^*)^3-\frac{1}{12}\lambda_{u}^2  (t^*)^4e^{\lambda_ut^*}$
is such that:

$$-\frac{1}{3}\lambda_u(t^*)^3-\frac{1}{12}\lambda_{u}^2  (t^*)^4e^{\lambda_ut^*}\geq 
\frac{2 t^*}{\lambda_{u}}+(t^*)^2+\frac{2}{\lambda_{u}^2}(1-e^{\lambda_{u}  t^*})$$

since

$\frac{2 t^*}{\lambda_{u}}+(t^*)^2+\frac{2}{\lambda_{u}^2}(1-e^{\lambda_{u}  t^*})$

$=\frac{2 t^*}{\lambda_{u}}+(t^*)^2+\frac{2}{\lambda_{u}^2}(-\lambda_u t^*-\frac{1}{2}\lambda_u^2 (t^*)^2-\frac{1}{6}\lambda_u^3(t^*)^3-\frac{1}{24}\lambda_u^4(t^*)^4 e^{\lambda_u\theta}$

$=\frac{2}{\lambda_u^2}(-\frac{1}{6}\lambda_u^3(t^*)^3-\frac{1}{24}\lambda_u^4(t^*)^4 e^{\lambda_u\theta})$ for some $0\leq \theta\leq t^*$

$=-\frac{1}{3}\lambda_u(t^*)^3-\frac{1}{12}\lambda_u^2(t^*)^4 e^{\lambda_u\theta}$

$\leq -\frac{1}{3}\lambda_u(t^*)^3-\frac{1}{12}\lambda_u^2(t^*)^4e^{\lambda_ut^*}.$\\

It follows:

$$e_0^2(e^{\lambda_{u}  t^*}-1)+\frac{C_{u}^2}{\lambda_{u}^2}(\frac{2 t^*}{\lambda_{u}}+(t^*)^2+\frac{2}{\lambda_{u}^2}(1-e^{\lambda_{u}  t^*}))\leq 0.$$

$$e_0^2e^{\lambda_{u}  t^*}+\frac{C_{u}^2}{\lambda_{u}^2}(\frac{2 t^*}{\lambda_{u}}+(t^*)^2+\frac{2}{\lambda_{u}^2}(1-e^{\lambda_{u}  t^*}))\leq e_0^2.$$



i.e.

$$(\delta_{e_0}^u(t^*))^2\leq e_0^2.$$

Hence: $\delta_{e_0}^u(t^*)\leq e_0.$ It remains to show:
$\delta_{e_0}^u(t)\leq e_0$ for $t\in [0,t^*]$.

Consider the 1rst and 2nd derivative $\delta'(\cdot)$ and
$\delta''(\cdot)$ of $\delta(\cdot)$.
We have:

$\delta'(t)=\lambda_{u}e_0^2e^{\lambda_{u}t}+\frac{C_{u}^2}{\lambda_{u}^2}(2t+\frac{2}{\lambda_{u}}-\frac{2}{\lambda_{u}}e^{\lambda_{u}t})$

$\delta''(t)=\lambda_{u}^2e_0^2e^{\lambda_{u}t}+\frac{C_{u}^2}{\lambda_{u}^2}(2-2e^{\lambda_{u}t}).$

Hence $\delta''(t)>0$ for all $t\geq 0$.
On the other hand, for $t=0$, $\delta'(t)=\lambda_u e_0^2<0$, and for $t$ sufficiently large, $\delta'(t)>0$. Hence, $\delta'(\cdot)$ is strictly increasing and has a unique root. 
It follows that the equation $\delta(t)=e_0$ 
has a unique solution $t^{**}$ for $t>0$. Besides,
$\delta(t)\leq e_0$ for $t\in [0,t^{**}]$, and $\delta(t)\geq e_0$ for 
$t\in [t^{**},+\infty)$.
Since we have shown: $\delta(t^*)\leq e_0$, it follows $t^*\leq t^{**}$ and
$\delta(t)\leq e_0$ for $t\in [0,t^{*}]$.

\hspace*{10cm} $\Box$
\end{proof}

\newpage
\section*{Appendix 2: Numerical results}
\begin{table}
\centering
\begin{tabular}{|c|c|c|c|}
   \hline
    Dimension & Extended mode length & $ \| y_i^{\pi^\varepsilon} (T) - y_i^f \| $ for $\sigma=1$ & $ \| y_i^{\pi^\varepsilon} (T) - y_i^f \| $ for $\sigma=0.5$   \\
    \hline
    $i=1\ (M_i=5)$ & 1 & 0.27642 & 0.33869 \\
    \cline{2-4} 
        & 2& 0.44496  & 0.39068\\
     \cline{2-4} 
        & 4& 0.15294  & 0.22024\\
    \hline 
    
    $i=2\ (M_i=10)$ & 1 & 0.39904  & 0.50251\\
    \cline{2-4} 
        & 2& 0.50092  & 0.58500\\
     \cline{2-4} 
        & 4& 0.16738 & 0.31440\\
        \hline
\end{tabular}
\caption{Value $\| y_i^{\pi^\varepsilon} (T) - y_i^f \|$
for $\sigma=1$ and $\sigma=0.5$ ($T=2$, $i=1,2$).}
\end{table}
\begin{table}
\centering
\begin{tabular}{|c|c|c|}
   \hline
     Extended mode length & $\|P {y}^{\pi^\varepsilon}_2(T)-y_1^f\|$ for $\sigma=1$ & $\|P {y}^{\pi^\varepsilon}_2(T)-y_1^f\|$ for $\sigma=0.5$  \\
     \hline
    1 &  0.67429 & 0.77322\\
\hline        
 2&  0.27501 & 0.72322\\
\hline        
4 & 0.31385 & 0.21481\\
        \hline
\end{tabular}
\caption{Projection value $\|P {y}^{\pi^\varepsilon}_2(T)-y_1^f\|$ for $\sigma=1$, $\sigma=0.5$ ($T=2$).}
\end{table}
%
\newpage
\begin{figure}
\centering
\begin{tabular}{cc}
Length 1, $M_1=5$ & Length 1, $M_2=10$ \\
\includegraphics[width=0.47\textwidth]{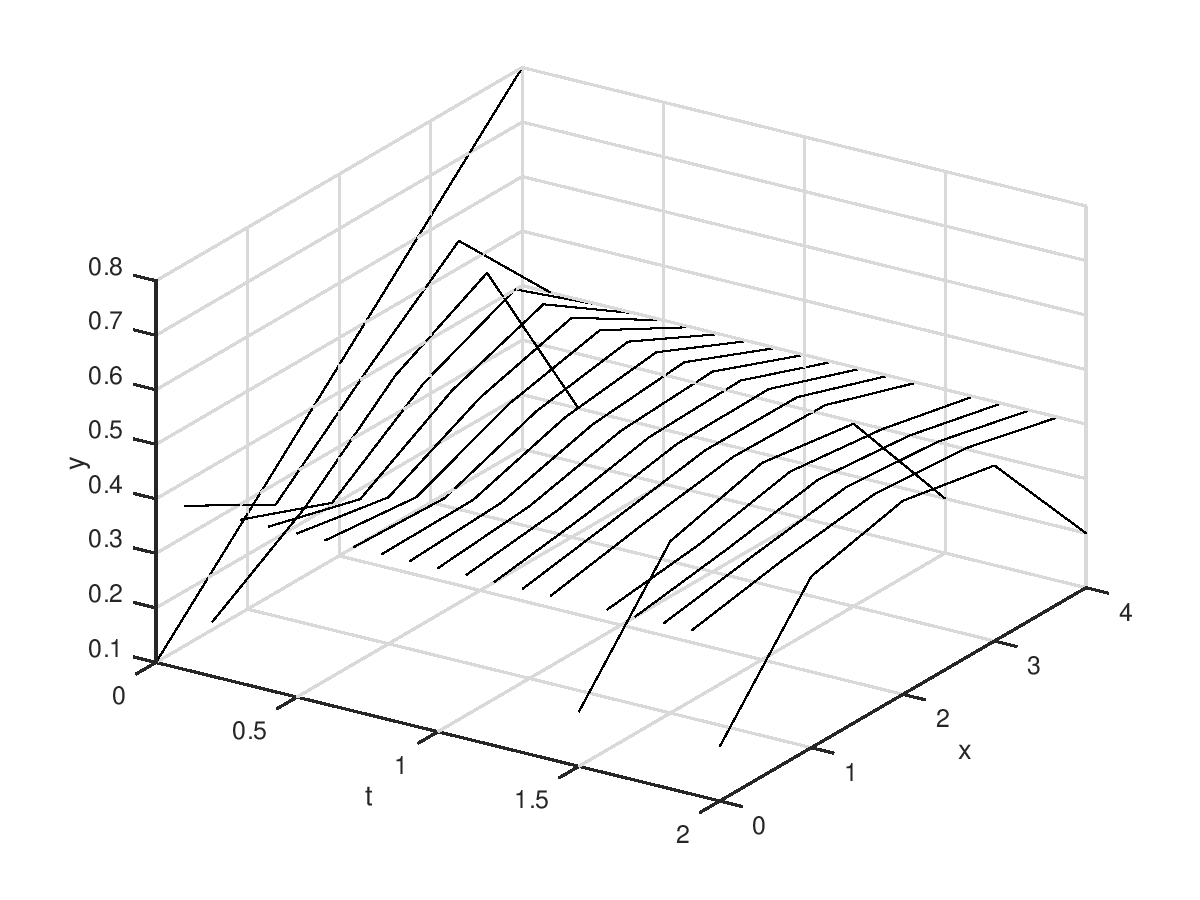} & \includegraphics[width=0.47\textwidth]{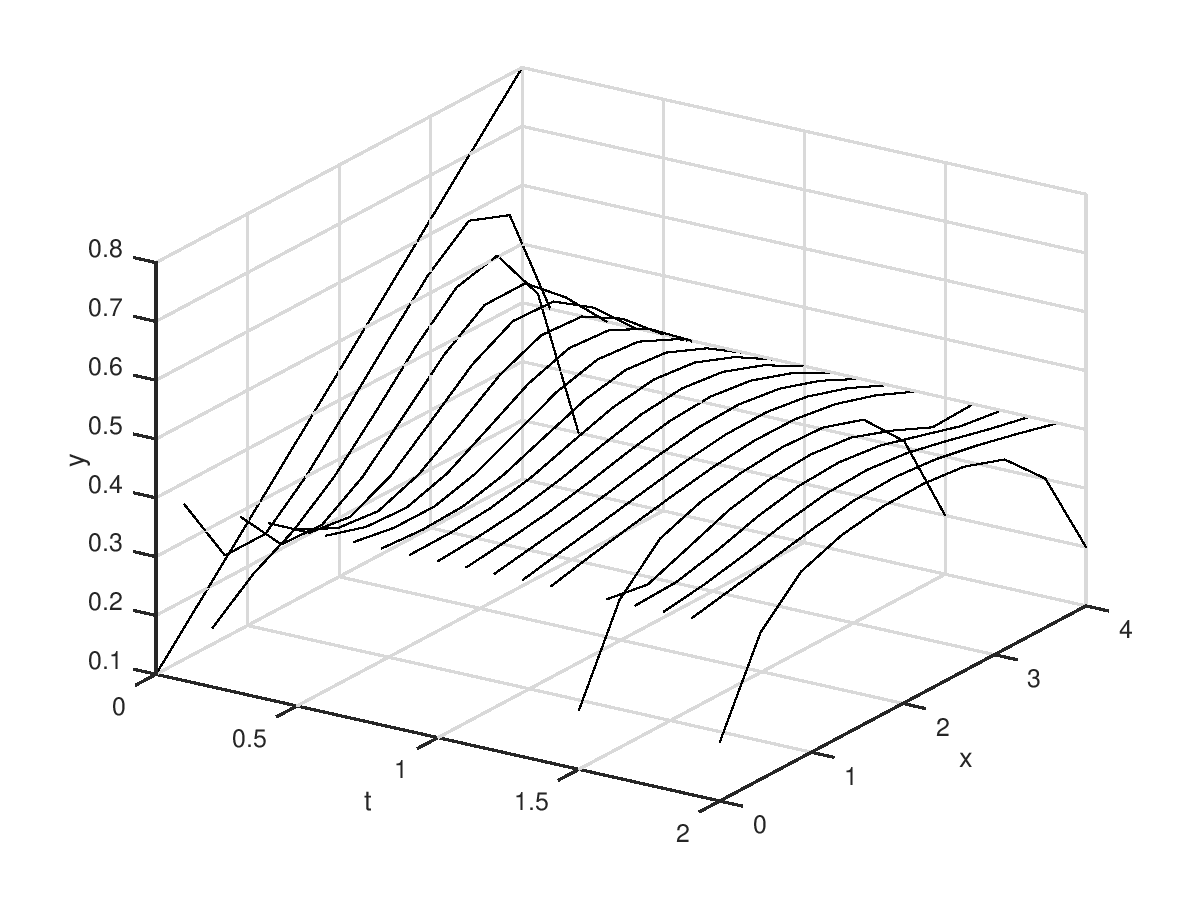} \\
Length 2, $M_1=5$ & Length 2, $M_2=10$ \\
\includegraphics[width=0.47\textwidth]{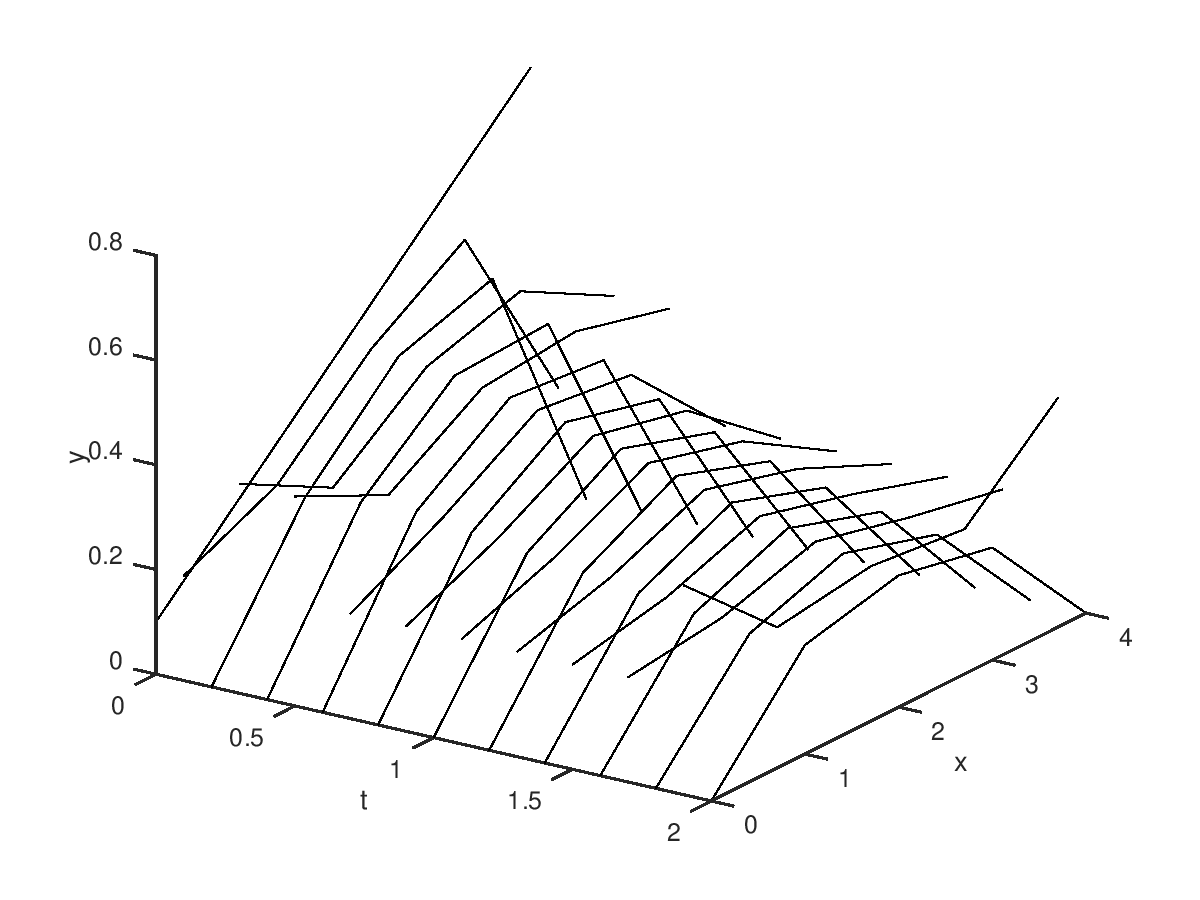} & \includegraphics[width=0.47\textwidth]{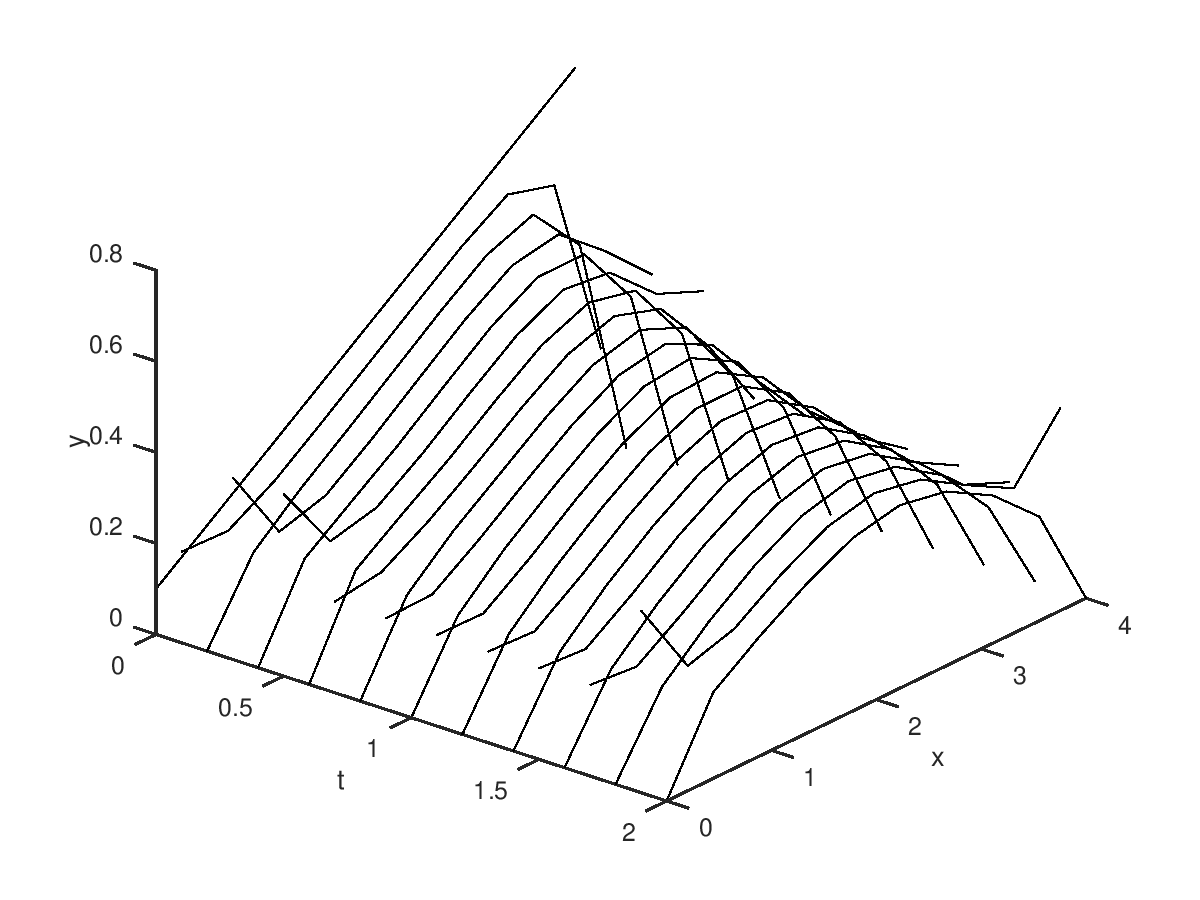} \\
Length 4, $M_1=5$ & Length 4, $M_2=10$ \\
\includegraphics[width=0.47\textwidth]{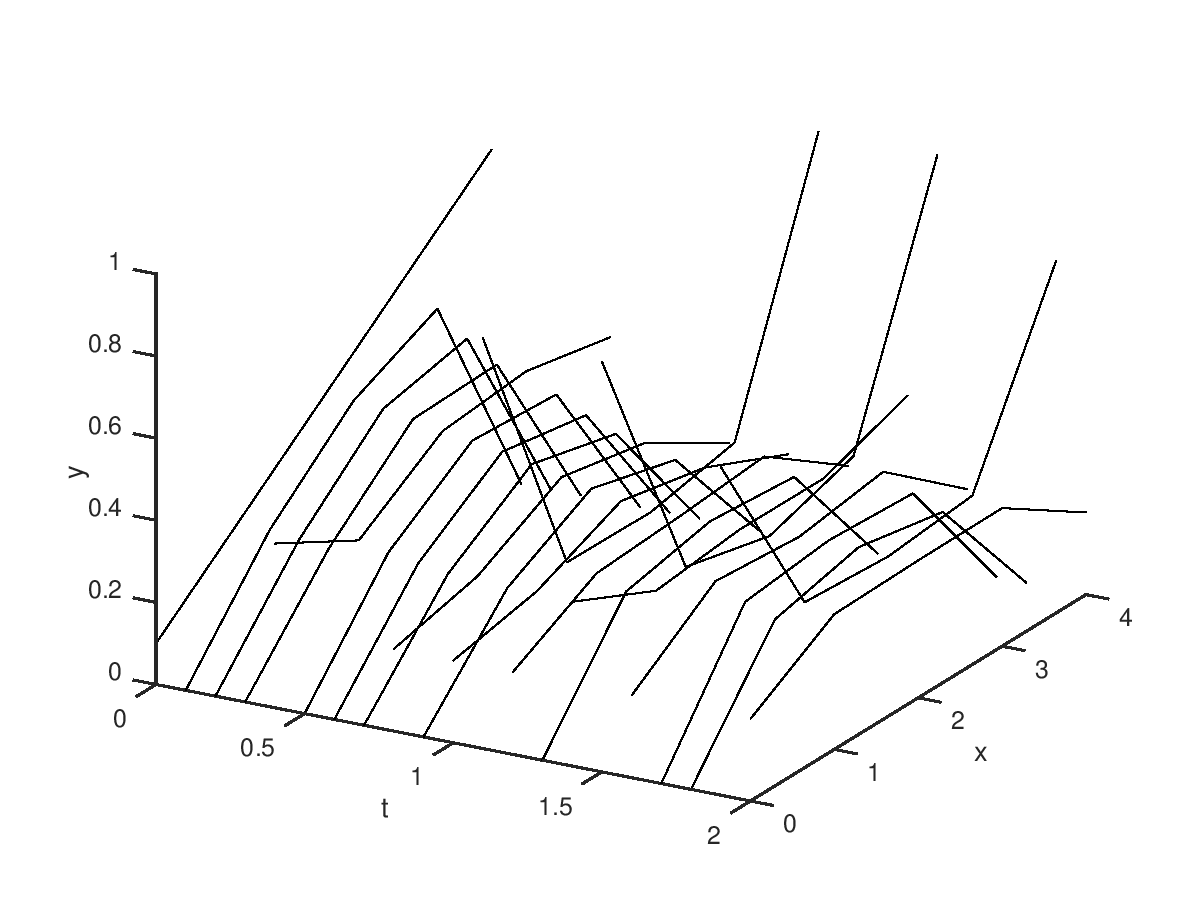} & \includegraphics[width=0.47\textwidth]{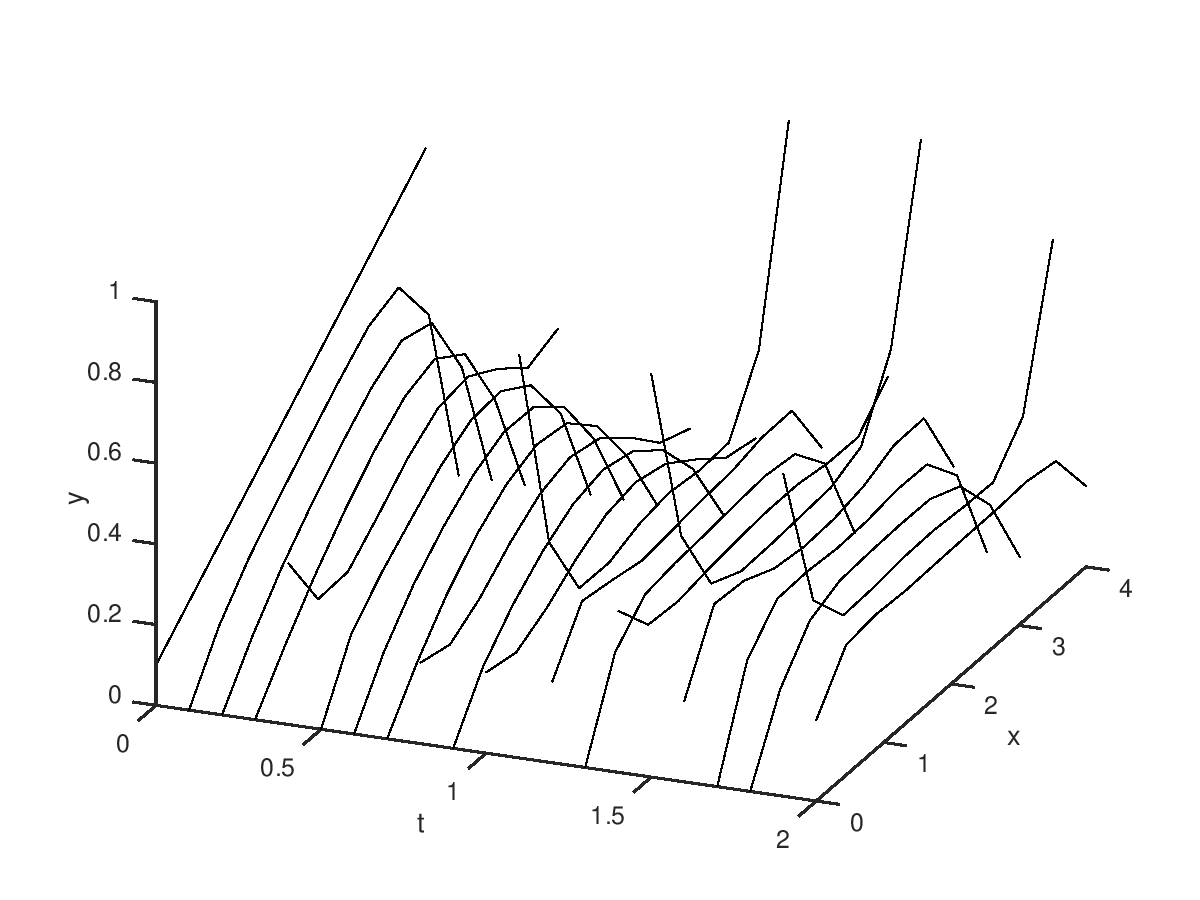}
\end{tabular}
\caption{Simulation of the controllers for $\sigma=1$.}
\end{figure}

\newpage

\begin{figure}
\centering
\begin{tabular}{cc}
Length 1, $M_1=5$ & Length 1, $M_2=10$ \\
\includegraphics[width=0.47\textwidth]{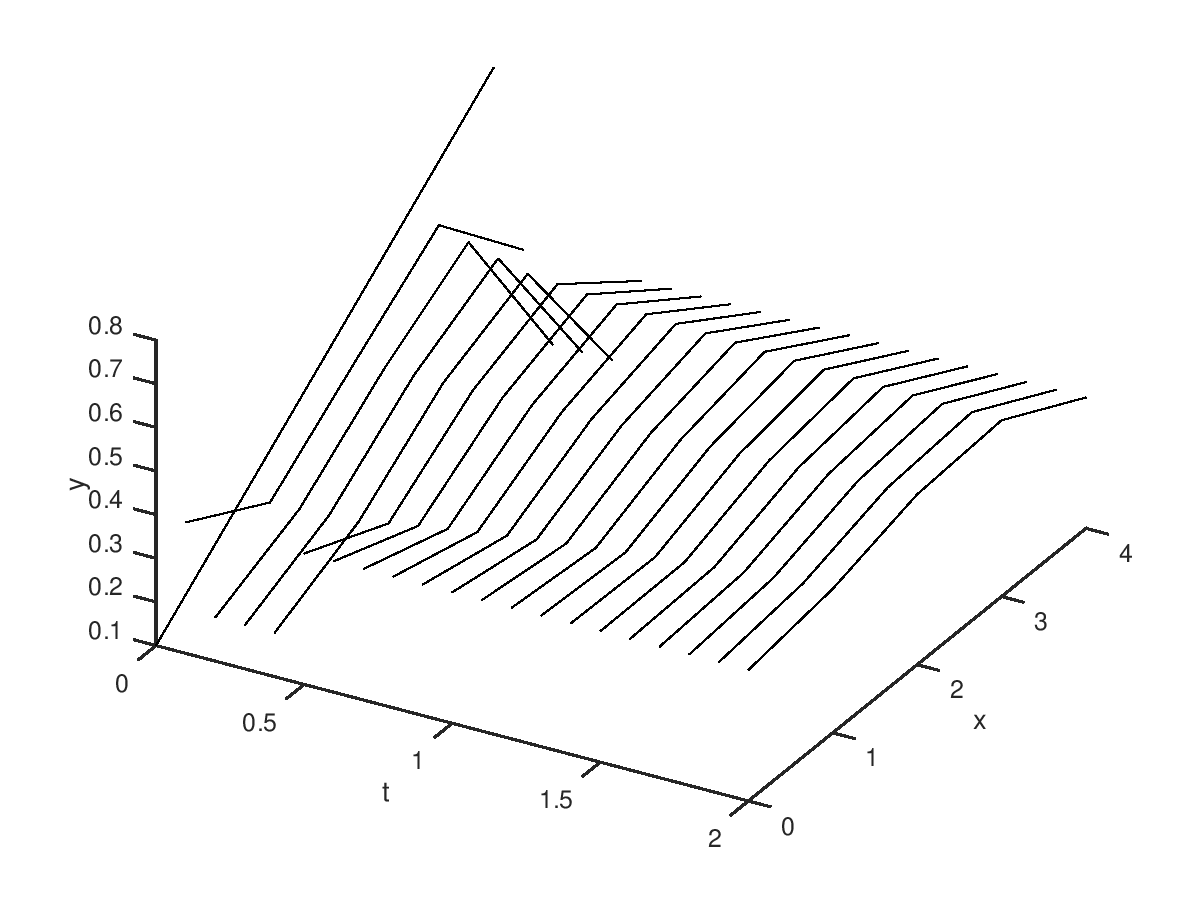} & \includegraphics[width=0.47\textwidth]{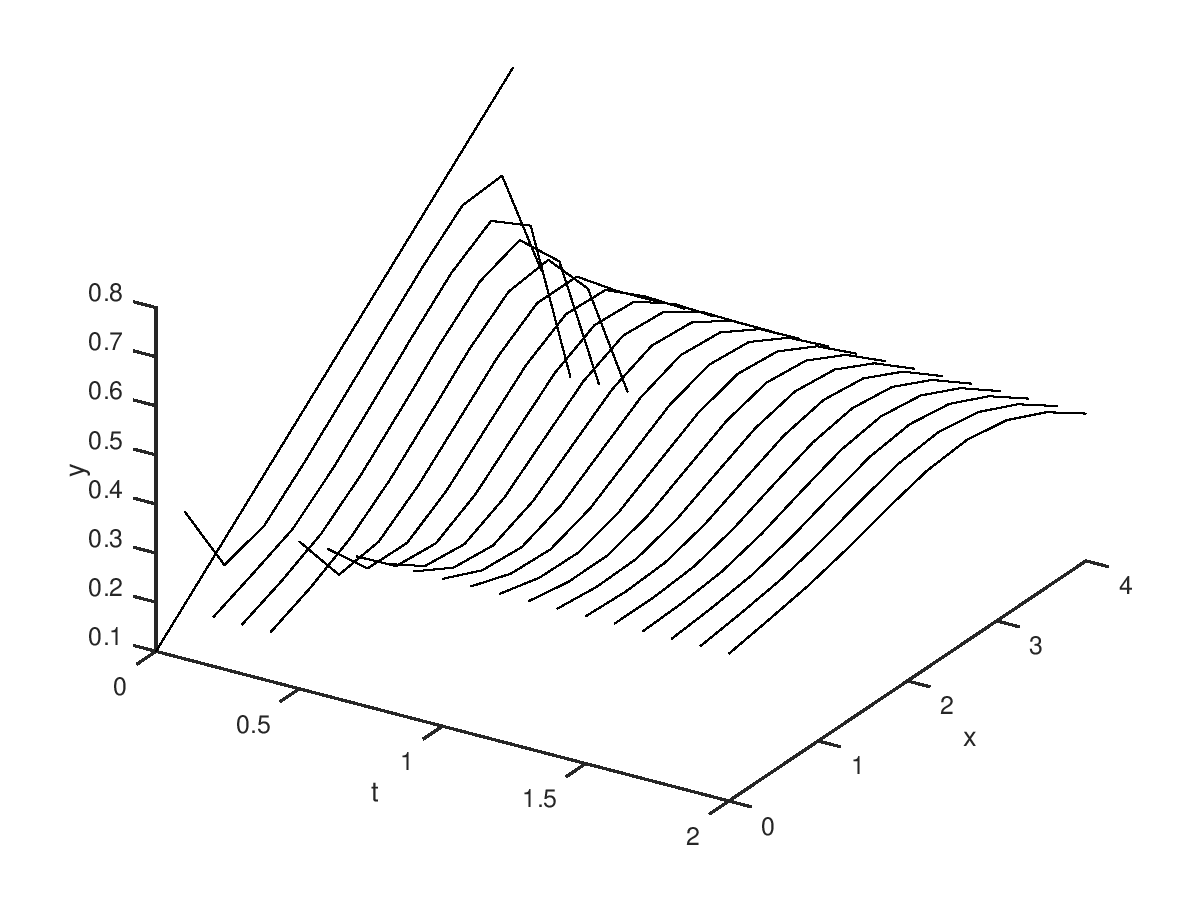} \\
Length 2, $M_1=5$ & Length 2, $M_2=10$ \\
\includegraphics[width=0.47\textwidth]{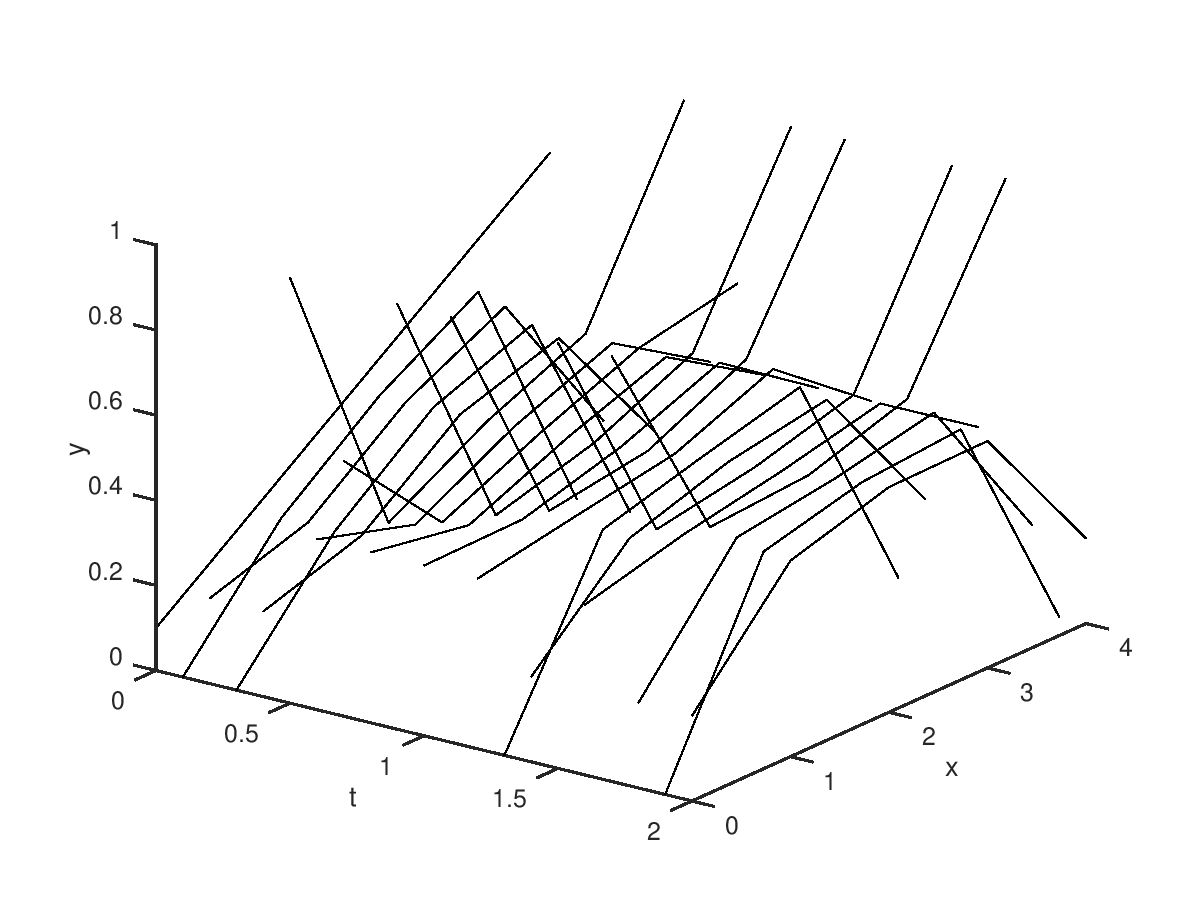} & \includegraphics[width=0.47\textwidth]{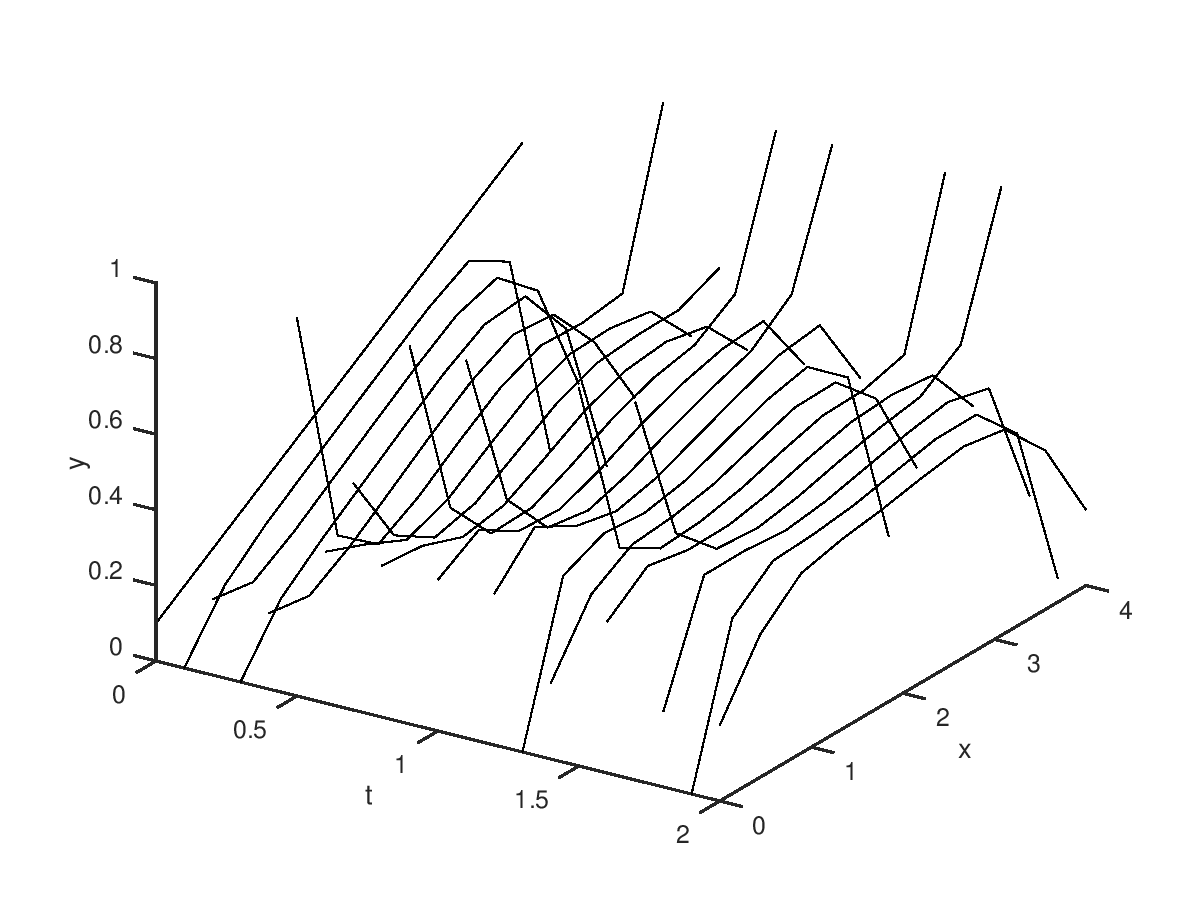} \\
Length 4, $M_1=5$ & Length 4, $M_2=10$ \\
\includegraphics[width=0.47\textwidth]{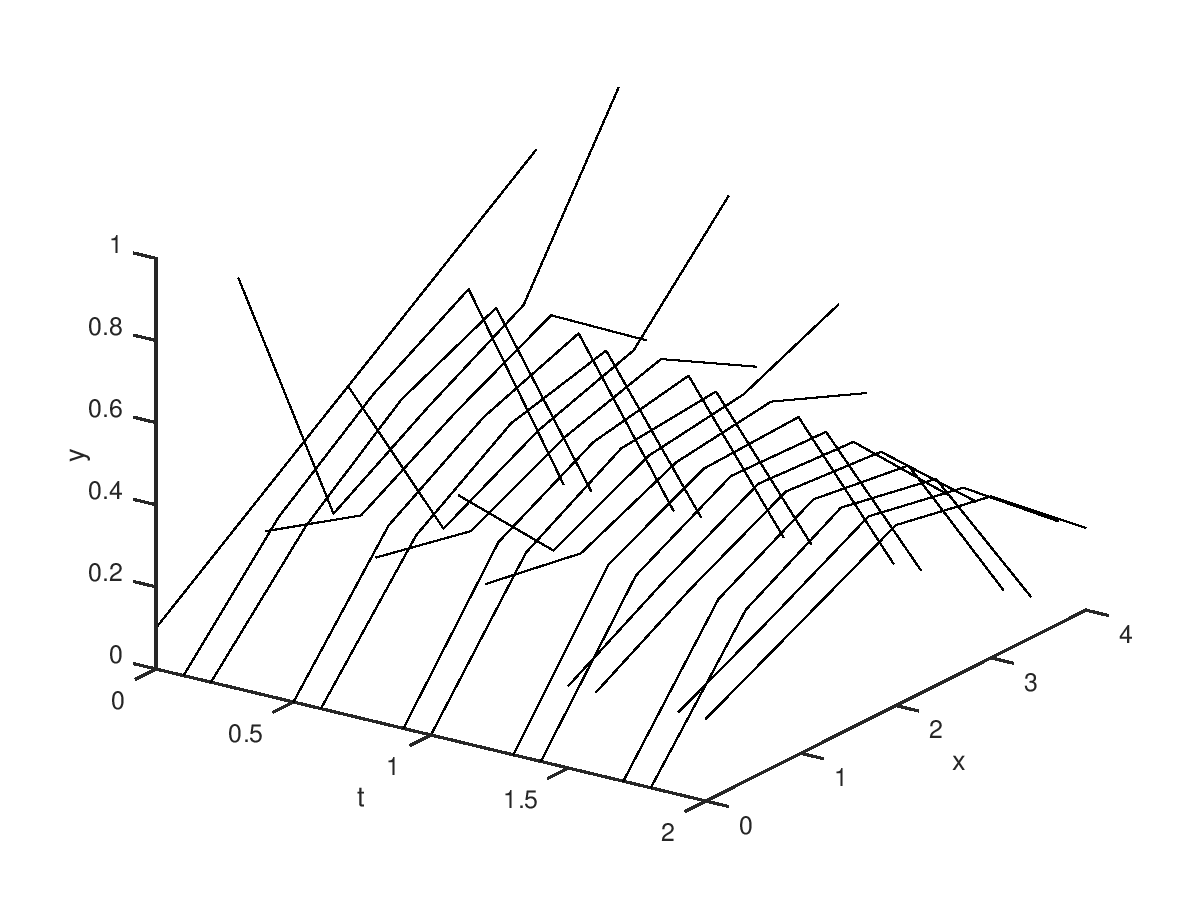} & \includegraphics[width=0.47\textwidth]{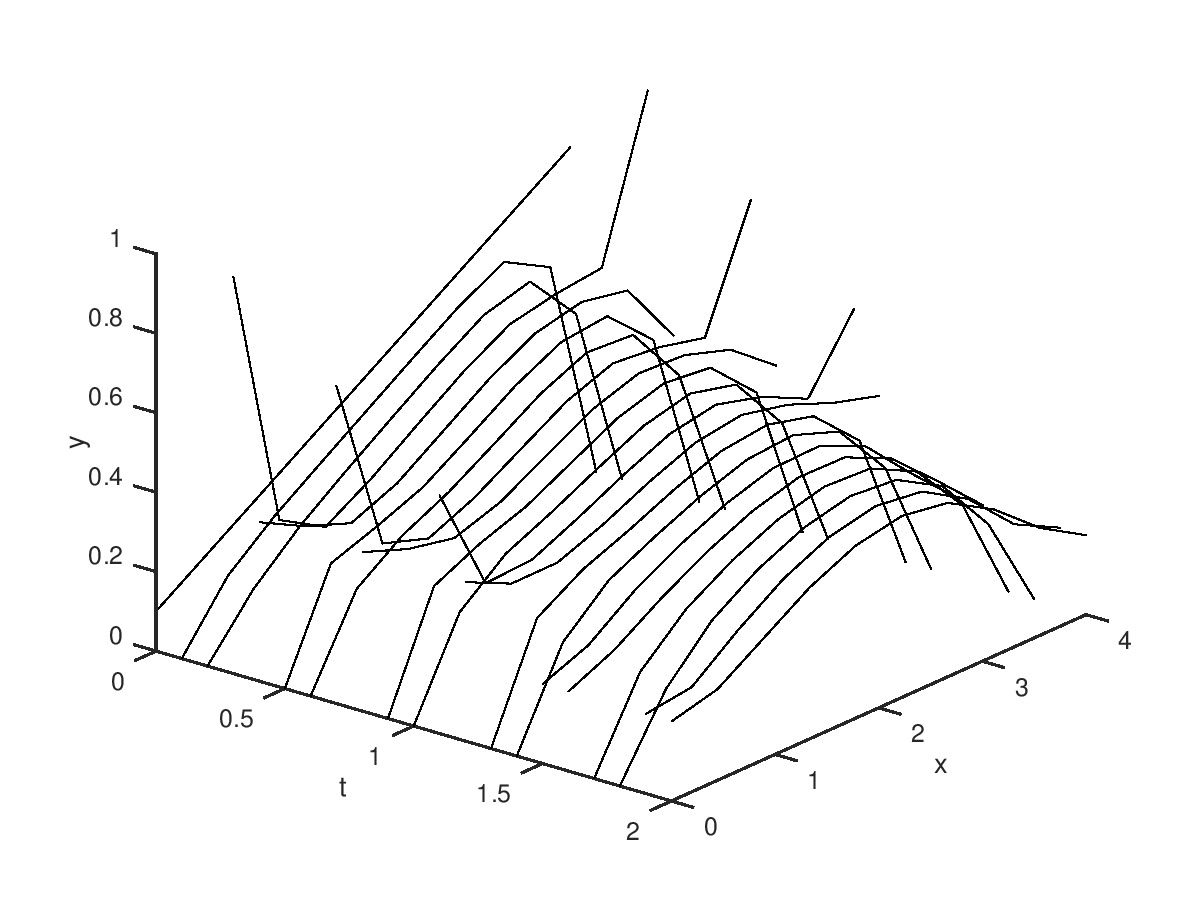}
\end{tabular}
\caption{Simulation of the controllers for $\sigma=0.5$.}
\end{figure}

\end{document}